\documentclass[11pt]{amsart}
\usepackage{amsmath,amssymb,hyperref}
\usepackage{amsfonts}
\usepackage{enumerate}
\usepackage{amsmath, amsthm}
\usepackage{amscd}
\input xy
\xyoption{all}
\hypersetup{pdfpagemode=FullScreen, colorlinks=true}

\newtheorem{thm}{Theorem}[section]

\newtheorem{prop}[thm]{Proposition}
\newtheorem{lemma}[thm]{Lemma}

\newtheorem{cor}[thm]{Corollary}

\theoremstyle{definition}
\newtheorem{defi}[thm]{Definition}
\newtheorem*{remark}{Remark}

\newtheorem*{setup}{Setup}

\newcommand{\bc}{{\mathbb C}}

\newcommand{\br}{{\mathbb R}}
\newcommand{\bq}{{\mathbb Q}}
\newcommand{\bz}{{\mathbb Z}}
\newcommand{\fl}{\mathcal{F}}
\newcommand{\ra}{\rightarrow}

\let \cal \mathcal
\begin{document}
\title[Bloch groups]
{Homological and Bloch invariants for $\bq$-rank one spaces and flag structures}
\date{}

\author{Inkang Kim}
\address{School of Mathematics,
KIAS, Heogiro 85, Dongdaemun-gu Seoul, 130-722, Republic of Korea}
\email{inkang@kias.re.kr}
\author{Sungwoon Kim}
\email{sungwoon@kias.re.kr}
\author{Thilo Kuessner}
\email{kuessner@kias.re.kr}
\begin{abstract}
We use group homology to define invariants in algebraic K-theory and in an analogue of the Bloch group for $\bq$-rank one lattices and for some other geometric structures. We also show that the Bloch invariants of CR structures and of flag structures can be recovered by a fundamental class construction.
\end{abstract}

\footnotetext[1]{I. Kim gratefully acknowledges the partial support
of NRF grant  (2010-0024171) and a warm support of IHES
during his stay.}

\maketitle

\tableofcontents
\section{Introduction}
Fundamental class constructions and Bloch invariants are by now a
classical theme in the topological study of hyperbolic
$3$-manifolds, going back to the work of Dupont and Sah on
scissors congruences and more recently the work of Neumann, Yang
and Zickert on Bloch invariants. For a hyperbolic $3$-manifold
$M=\Gamma\backslash \mathbb H^3$ with $\Gamma\subset
\mathrm{PSL}(2,\bc)$ one can on the one hand consider its
$\mathrm{PSL}(2,\bc)$-fundamental class
$[M]_{\mathrm{PSL}(2,\bc)}$, that is the image of the fundamental
class $[M]\in H_3\left(M\right)\cong H_3(\Gamma)$ in
$H_3(\mathrm{PSL}(2,{\bc}))$, and on the other hand one can use
ideal triangulations or more generally degree one ideal
triangulations to define an invariant $\beta(M)$ in the Bloch
group ${\mathcal{B}}(\bc)$. In \cite{NY} it was shown that one can
recover the volume and the Chern-Simons invariant mod $\bq$ from
$\beta(M)$. (In later work Neumann constructed an invariant in an
extended Bloch group, from which one can recover the Chern-Simons
invariant mod $\bz$.)

The approach via ideal triangulations is better suited for doing
practical calculations, see for example \cite{NY}. On the other
hand the fundamental class approach is useful for theoretical
considerations, e.g.\ to study the behaviour of hyperbolic volume
under cut and paste in \cite{Ku2}. By the Bloch-Wigner Theorem
(proved in more generality by Dupont-Sah in \cite[Appendix A]{ds})
there is an isomorphism
$$H_3(\mathrm{PSL}(2,\bc),\bz)/Torsion\cong {\mathcal{B}}(\bc),$$
and this isomorphism sends $[M]_{\mathrm{PSL}(2,\bc)}$ to
$\beta(M)$. (One may pictorially think of a triangulation whose
vertices are moved to infinity to produce an ideal triangulation.
In some weak sense this picture can be made precise, see
\cite{Kue}.) In particular the Bloch invariant is determined by
the $\mathrm{PSL}(2,\bc)$-fundamental class.

The construction of the Bloch invariant was generalized to
higher-dimensional hyperbolic manifolds in \cite[Section 8]{NY}.
On the other hand Goncharov \cite[Section 2]{G} generalised the
fundamental class construction to get - associated to an
odd-dimensional hyperbolic manifold $M^{2k-1}$ and a spinor
representation $\mathrm{SO}(2k-1,1)\rightarrow \mathrm{GL}(n,\bc)$
- an element in $H_{2k-1}(\mathrm{GL}(n,\bc))$ such that
application of the Borel class recovers (a fixed multiple of) the
volume. In \cite[Section 3]{G} he also used ideal triangulations
to construct an extension $m(M^{2k-1}) \in
Ext^1_{{\mathcal{M}}_{\bq}} (Q(0);Q(k))$ in the category of mixed
Tate motives over ${\bq}$ and thus an element in
$K_{2k-1}(\overline{\bq})\otimes\bq$ according to Beilinson's
description of K-theory of fields.
(In degree 3 one has $K_3^{ind}(\bc)\otimes\bq={\mathcal{B}}(\bc)\otimes\bq$ and one recovers the Bloch invariant from this $K$-theoretic approach.)

In \cite{Ku} the third-named author generalized Goncharov's
(first) construction to finite-volume locally symmetric spaces of
noncompact type $M=\Gamma\backslash G/K$ which are either closed
or of $\br$-rank one. To each representation $\rho:G\rightarrow
\mathrm{SL}(n,\bc)$ the construction yields an element in
$H_*(\mathrm{SL}(n,\overline{\bq}))$ or, after a suitable
projection, an element in $$PH_*(\mathrm{GL}(\overline{\bq}))\cong
K_*(\overline{\bq})\otimes\bq$$ such that application of the Borel
class (to either of the two elements) yields a multiple $c_\rho
\mathrm{Vol}(M)$ of the volume $\mathrm{Vol}\left(M\right)$. The
factor $c_\rho$ depends only on $\rho$, in particular one can
recover the volume if $c_\rho\not=0$. Moreover, \cite[Section
3]{Ku} provides a complete list of fundamentals representations
$\rho:G\rightarrow \mathrm{SL}(n,\bc)$ with $c_\rho\not=0$, for
example one has $c_\rho\not=0$ whenever $dim\left(G/K\right)\equiv
3\ mod\ 4$ and $\rho\not=id$. However, in the noncompact case, the
only $\br$-rank one examples with $c_\rho\not=0$ were
(odd-dimensional) real-hyperbolic manifolds.

In this paper we further generalize Goncharov's construction to
${\bq}$-rank one spaces. That means, for a $\bq$-rank one locally
symmetric space of noncompact type $M=\Gamma\backslash G/K$ we
construct elements $$\overline{\gamma}(M)\in
H_*(\mathrm{SL}(n,\overline{\bq}))$$ and $$\gamma(M)\in
K_*(\overline{\bq})\otimes\bq$$ such that application of the Borel
class yields again $c_\rho \mathrm{Vol}(M)$. Compare
\hyperref[prop:preimage]{Proposition \ref*{prop:preimage}} and
\hyperref[thm:Ktheory]{Theorem \ref*{thm:Ktheory}} for the precise
statements. Thus one can get many non-compact non-hyperbolic
examples with nontrivial invariants.

The fundamental class construction shall be useful for deriving
general results about the relation of topology and volume. For
practical computations however the Bloch group approach appears to
be more feasible, not only in 3-dimensional hyperbolic geometry
\cite{NY} but also in the study of CR-structures \cite{FW} or flag
structures \cite{bfg}.

In the 3-dimensional hyperbolic case, Neumann-Yang constructed the
Bloch invariant \cite[Definition 2.5]{NY} in the so-called
pre-Bloch group ${\mathcal{P}}(\bc)$ and then proved in
\cite[Theorem 6]{NY} that it actually belongs to the Bloch group
${\mathcal{B}}(\bc)\subset {\mathcal{P}}(\bc)$. The pre-Bloch
group ${\mathcal{P}}(\bc)$ satisfies a natural isomorphism
$${\mathcal{P}}(\bc)\cong H_3(C_*(\partial_\infty G/K)\otimes_{\bz
G}\bz)$$ for $G/K=\mathrm{SL}(2,\bc)/\mathrm{SU}(2)=\mathbb H^3$.
Thus it is natural to define a Bloch invariant of
higher-dimensional locally symmetric spaces $M^d=\Gamma\backslash
G/K$ (and representations $\rho:G\rightarrow \mathrm{SL}(n,\bc)$)
as an element in $$H_d(C_*(\partial_\infty
\mathrm{SL}(n,\bc)/\mathrm{SU}(n))\otimes_{\bz
\mathrm{SL}(n,\bc)}\bz).$$ In \cite{Kue} this was done for
${\br}$-rank one symmetric spaces and it was shown that the Bloch
invariant is the image of the Goncharov invariant
$\overline{\gamma}\left(M\right)$ under a naturally defined
evaluation homomorphism which generalizes the homomorphism from
the Bloch-Wigner Theorem. The construction of the generalized
Bloch invariant uses proper ideal fundamental cycles since the
existence of ideal triangulations is unclear in general. The proof
of well-definedness of the Bloch invariant (i.e.\ indepencence
from the chosen proper ideal fundamental cycle, \cite[Lemma
3.4.1]{Kue}) was building on the equality $H_d(C_*(\partial_\infty
G/K)\otimes_{\bz \Gamma}\bz)\cong \bz$ for lattices $\Gamma\subset
G$, which in the $\mathbb R$-rank one case can be proved by an
immediate generalization of the results of Neumann-Yang (who
proved this equality in \cite{NY} for hyperbolic $3$-space).
However it is unclear how to generalize this argument to the
higher rank case. Therefore we avoid this point in Section 8 by
directly defining the Bloch invariant
$$\beta(M)\in H_d(C_*(\partial_\infty \mathrm{SL}(n,\bc)/\mathrm{SU}(n))\otimes_{\bz
\mathrm{SL}(n,\bc)}\bz)$$ for locally symmetric spaces (either closed or
of $\bq$-rank one) as the image of the Goncharov invariant
$\overline{\gamma}(M)$ under the evaluation homomorphism.

In Section 9 we consider a similar construction for convex
projective manifolds. We hope to further exploit this in other
papers.

In Section 10 we prove that also the Bloch invariants of CR
structures (as defined by Falbel-Wang in \cite{FW}) and of flag
structures (as defined by Bergeron-Falbel-Guilloux in \cite{bfg})
can be recovered from a fundamental class construction. We apply
this to prove that these Bloch invariants are preserved under
certain cut-and-paste operations.

\section{Basics on Group (co)homology}
\subsection{Group homology} For a topological group $G$, let
$G^\delta$ denote the group with discrete topology. Let $BG^\delta$ denote
the simplicial set whose $k$-simplices are $k$-tuples
$(g_1,\ldots,g_k)$ with a natural boundary operator $\partial$:
{\setlength\arraycolsep{2pt}
\begin{eqnarray*}
\partial(g_1,\ldots,g_k) = (g_2,\ldots,g_k) &+& \sum_{i=1}^{k-1}(-1)^i(g_1,\ldots,g_ig_{i+1},\ldots,g_k) \\
&+& (-1)^k(g_1,\ldots,g_{k-1}).
\end{eqnarray*}}
It forms a chain complex $C_*^{simp}(BG^\delta)$ of $BG^\delta$ whose
homology with a coefficient ring $R$ is defined as the group
homology
$$H_*(G,R)=H_*^{simp}(BG^\delta,R)=H_*^{simp}(C_*(BG^\delta\otimes_\bz R),
\partial\otimes 1).$$
Throughout the paper, $BG$ will be understood as $BG^\delta$.

Let $M$ be a Riemannian manifold of nonpositive sectional curvature and
$x_0\in M$, a lift $\tilde x_0\in\widetilde M$ of $x_0$ be fixed. Any {\em ordered} tuple of vertices in $\widetilde{M}$ determines a unique straight simplex. A singular simplex $\sigma\in C_*(M)$ is
straight if some (hence any) lift $\tilde\sigma \in C_*(\widetilde M)$ is straight.
Let $C_*^{str,x_0}(M)$ be the chain complex of straight simplices
with all vertices $x_0$.

Set $\Gamma=\pi_1(M,x_0)$. Then there are two canonical homomorphisms $\Psi:C_*^{simp}(B\Gamma)\ra
C_*^{str,x_0}(M)$ defined by
$$ \Psi(g_1,\ldots,g_k)=\pi(str(\tilde x_0, g_1\tilde x_0, g_1g_2\tilde x_0,\ldots,g_1\cdots g_k \tilde x_0))$$
and $\Phi:C_*^{str,x_0}(M)\ra C_*^{simp}(B\Gamma)$ defined by
$$ \Phi(\sigma)=([\sigma|_{\zeta_1}],\ldots, [\sigma|_{\zeta_k}])$$ where
$e_0,\ldots,e_k$ are the vertices of the standard simplex
$\Delta^k$, $\zeta_i$ is the standard sub-1-simplex with $\partial
\zeta_i = e_i -e_{i-1}$, each $[\sigma|_{\zeta_i}]\in
\pi_1(M,x_0)=\Gamma$ is the homotopy class of $\sigma|_{\zeta_i}$
and $\pi : \widetilde M \ra M$ is the universal covering map of
$M$. It is easy to show that $\Psi$ and $\Phi$ are chain
isomorphisms inverse to each other.

The inclusion
$$i:C_*^{str,x_0}(M)\subset C_*(M)$$ and the straightening (\cite[Section 2.1]{Ku})
$$str:C_*(M)\ra C_*^{str,x_0}(M)$$ are chain homotopy inverses.
Hence we have a chain homotopy equivalence, called the
\emph{Eilenberg-MacLane map}
$$EM: C_*^{simp}(B\Gamma)\ra C_*(M).$$ We will frequently use the induced isomorphism
$$EM_*^{-1}=\Phi_*\circ str_*: H_*(M,\bz)\ra H_*^{simp}(B\Gamma, \bz).$$
The geometric realization $|B\Gamma|$ is a $K(\Gamma,1)$, thus there is a classifying map $h^M:M\ra|B\Gamma|$ which
induces an isomorphism on $\pi_1$ level. The inclusion map of
simplices $i:C_*^{simp}(B\Gamma)\ra C_*(|B\Gamma|)$, induces an isomorphism
$$i_*:H_*^{simp}(B\Gamma,\bz)\ra H_*(|B\Gamma|,\bz)$$ such that
$h^M_*=i_*\circ EM_*^{-1}$ if $M$ is aspherical and of the homotopy type of a CW complex (which is always the case for Riemannian manifolds of nonpositive sectional curvature).

For a commutative ring $A\subset \bc$ with unit, let
$\mathrm{GL}(A)=\cup_{n=1}^\infty \mathrm{GL}(n,A)$ be the increasing union,
and $|B\mathrm{GL}(A)|$ its classifying space as above.

Let $\rho:\Gamma\ra \mathrm{GL}(A)$ be a
representation. This induces $B_\rho:B\Gamma\ra B\mathrm{GL}(A)$
and $|B_\rho|:|B\Gamma|\ra |B\mathrm{GL}(A)|$. The composition
$$ \xymatrixcolsep{3pc}\xymatrix{H_*(M,\bq)\ar[r]^-{EM_*^{-1}} & H_*^{simp}(B\Gamma,\bq)\ar[r]^-{(B_\rho)_*} &H_*^{simp}(B\mathrm{GL}(A),\bq)
}$$ induces a map $$(H_\rho)_*: H_*(M,\bq)\ra
H_*^{simp}(B\mathrm{GL}(A),\bq).$$ If $M$ is a closed, oriented
and connected $d$-dimensional manifold, $(H_\rho)_d[M]$ will play
an important role for us where $[M]$ is the fundamental class in $
H_d(M,\bq) \cong \bq$.

\subsection{Volume class and Borel class}\label{volume}
Let $G$ be a noncompact semisimple, connected Lie group. Let
$X=G/K$ be the associated symmetric space of dimension $d$ with a
maximal compact subgroup $K$ of $G$. Let us denote by
$H^*_c(G,\br)$ the continuous cohomology of $G$. The comparison
map $comp: H^*_c(G,\br)\ra H^*_{simp}(BG,\br)$ is defined by the
cochain map
$$comp(f)(g_1,\ldots,g_k)=f(1,g_1,g_1g_2,\ldots,g_1g_2\cdots
g_k)$$ for a $G$-invariant cochain $f:G^{k+1} \ra \br$. Fix a
point $x\in X$. The volume class $v_d\in H_c^d(G,\br)$ is defined
by the cocycle $$\nu_d(g_0,\ldots,g_d)=algvol(str(g_0x,\ldots,g_d
x)):=\int_{str(g_0x,\ldots,g_d x)} dvol_X, $$ where $dvol_X$ is
the $G$-invariant Riemannian volume form on $X$, and $str$ is the
geodesic straightening of a simplex with those vertices. (That is $str(g_0x,\ldots,g_d x)$ is the unique straight simplex with the given {\em ordered} set of vertices.) Under the comparison map
$$comp(v_d)(g_1,\ldots,g_d)=algvol(str(x,g_1x,\ldots,g_1\cdots
g_d x)).$$

Then it is not difficult to show that if $N=\Gamma\backslash X$ is
a closed locally symmetric space, with $j:\Gamma \ra G$ the
inclusion, then $$\mathrm{Vol}(N)= \langle comp(v_d), B_j \circ
EM^{-1}_d[N] \rangle.$$ For a detailed proof of this, see
\cite[Theorem 1]{Ku}.

Let $\mathfrak{g}$ and $\mathfrak{k}$ be the Lie algebra of $G$
and $K$ respectively. If
$\mathfrak{g}=\mathfrak{k}\oplus\mathfrak{p}$ is a Cartan
decomposition, then the Lie algebra of the compact dual $G_u$ of
$G$ is $$\mathfrak{g}_u=\mathfrak{k}\oplus i \mathfrak{p}.$$ Note
that the relative Lie algebra cohomology
$H^*(\mathfrak{g},\mathfrak{k})$ is the cohomology of the complex
of $G$-invariant differential forms on $G/K$ and there is the Van Est
isomorphism
$$\mathcal J: H^*_c(G,\br)\ra H^*(\mathfrak{g},\mathfrak{k}).$$

Let $I_S^k(G_u)$ and $I_A^k(G_u)$ be
the space of $ad$-invariant symmetric and
antisymmetric multilinear $k$-forms on $\mathfrak{g}_u$.
There are isomorphisms $$\Phi_A:I_A^k(G_u)\ra H^k(G_u,\br),$$ and
the Chern-Weil isomorphism $$\Phi_S:I_S^k(G_u)\ra
H^{2k}(BG_u,\br),$$ see for example \cite[Section 5]{bur}.

When  $G_u=\mathrm{U}(n)$ is a unitary group,
there is $$\mathrm{Tr}_k(A_1,\ldots,A_k)=\frac{1}{(2\pi
i)^k}\frac{1}{k!}\sum_{\sigma\in S_k}
\mathrm{Tr}(A_{\sigma(1)}\cdots A_{\sigma(k)})$$ so that
$$C_k=\Phi_S(\mathrm{Tr}_k)\in H^{2k}(B\mathrm{U}(n),\br)$$ is the $2k$-th component of
the \emph{universal Chern character}.

For the fibration $G_u\ra EG_u\ra BG_u$ and an associated
transgression map $\tau$ from a subspace of $H^{2k-1}(G_u,\bz)$ to
$H^{2k}(BG_u,\bz)/ker(s)$, where $s$ is the suspension homomorphism
\cite{Bo1}, there is a homomorphism \cite{Cartan}
$$R:I_S^k(G_u)\ra I_A^{2k-1}(G_u)$$ such that $\tau \circ
\Phi_A\circ R=\pi\circ \Phi_S$ where $\pi:H^{2k}(BG_u,\bz)\ra
H^{2k}(BG_u,\bz)/ker(s)$. Then the Borel class is
$$b_{2k-1}=\Phi_A(R(\mathrm{Tr}_k))\in
H^{2k-1}(\mathrm{U}(n),\br)=H^{2k-1}(\mathfrak{u}(n),\br).$$ The
last equality holds since $\mathrm{U}(n)$ is a compact manifold,
$H^*(\mathrm{U}(n),\br)$ is the de Rham cohomology of
$\mathrm{U}(n)$, and by averaging it is isomorphic to the
cohomology of the complex of  $\mathrm{U}(n)$-invariant differential
forms on $\mathrm{U}(n)$. On the other hands, since
$(\mathrm{GL}(n,\bc))_u=\mathrm{U}(n)\times \mathrm{U}(n)$ and via
Van Est isomorphism {\setlength\arraycolsep{2pt}
\begin{eqnarray*}
H^*_c(\mathrm{GL}(n,\bc),\br) &=& H^*(\mathfrak{gl}(n,\bc), \mathfrak{u}(n)) \\
&=& H^*(\mathfrak u(n)\oplus \mathfrak u(n),\mathfrak u(n))=H^*(\mathfrak u(n),\br)
\end{eqnarray*}}
we may consider
$$b_{2k-1}\in H^{2k-1}_c(\mathrm{GL}(n,\bc),\br).$$

\subsection{Volume for compact locally symmetric manifolds}

It is a standard fact that for $d=dim(G/K)$, $H_c^d(G,\br) \cong \br$ via Van Est
isomorphism. If $\Gamma$ is a uniform lattice in $G$, then
$$H^d(\Gamma\backslash G/K,\br)\cong H^d(\Gamma,\br)\cong H^d_c(G,\br) \cong \br.$$
Hence the volume class $v_d$ as defined in Section 2.2 can be viewed as a
generator of $H^d(\Gamma,\br)$.

\begin{prop}\label{evaluation}
For a symmetric space $G/K$ of noncompact type with odd dimension
$d=2m-1$, and a representation $\rho: \Gamma\ra \mathrm{GL}(n,\bc)$
with a closed manifold $N=\Gamma\backslash G/K$, there exists a
constant $c_\rho \in \br$ such that
$$ \langle
 comp(b_d),  H_\rho[N]\rangle =c_\rho \mathrm{Vol}(N).$$
If $\rho:\Gamma\ra\mathrm{GL}(n,\bc)$ factors over a representation $\rho_0:G\ra\mathrm{GL}(n,\bc)$, then $c_\rho$ depends only on $\rho_0$.
\end{prop}
\begin{proof}
A representation $\rho :\Gamma \ra \mathrm{GL}(n,\bc)$ induces a
homomorphism $$\rho_c^* : H^d_c(\mathrm{GL}(n,\bc),\br)\ra
H^d(\Gamma,\br).$$ Since $\rho^*_c(b_d)\in H^d(\Gamma,\br)=\br
\cdot v_d$, there is a constant $c_\rho \in \br$ such that
$\rho_c^*(b_d)=c_\rho v_d$. Hence for a fundamental cycle
$\sum_{i=1}^l a_i\sigma_i$ of $N$, {\setlength\arraycolsep{2pt}
\begin{eqnarray*}
\langle comp(b_d), H_\rho[N]\rangle &=& \langle comp(b_d), B_\rho \circ EM^{-1}_d [N] \rangle \\
&=& \langle comp(\rho^*_c b_d), \Phi_*\circ str_*[N]\rangle \\
&=& c_\rho \sum_{i=1}^l a_i \langle comp(\nu_d), \Phi_*(str\sigma_i) \rangle \\
&=& c_\rho \sum_{i=1}^l a_i \cdot algvol(str \sigma_i) \\
&=& c_\rho \mathrm{Vol}(N)
\end{eqnarray*}}
We refer the reader to \cite[Theorem 2]{Ku} for the proof of the second claim.
\end{proof}

\section{$\bq$-rank $1$ locally symmetric spaces}

In this section, we consider only  $\bq$-rank $1$ lattices $\Gamma\subset G$.
We first collect some definitions and results about $\bq$-rank $1$ lattices.

\subsection{Arithmetic lattices}

Let $G$ be a noncompact, semisimple Lie group with trivial center and no compact factors.
Then one may define arithmetic lattices in the following way.

\begin{defi}\label{def:arithmetic}
A lattice $\Gamma$ in $G$ is called \emph{arithmetic} if there are
\begin{itemize}
\item[(1)] a semisimple algebraic group $\mathbf{G} \subset \mathrm{GL}(n,\mathbb{C})$ defined over $\bq$ and
\item[(2)] an isomorphism $\varphi : \mathbf{G}(\br)^0 \ra G$
\end{itemize}
such that $\varphi(\mathbf{G}(\mathbb{Z}) \cap \mathbf{G}(\br)^0)$ and $\Gamma$ are commensurable.
\end{defi}

It is well-known due to Margulis \cite{Ma91} that all irreducible
lattices in higher rank Lie groups are arithmetic. The $\bq$-rank
of a semisimple algebraic group $\mathbf{G}$ is defined as the
dimension of a maximal $\bq$-split torus of $\mathbf{G}$. For an
arithmetic lattice $\Gamma$ in $G$, $\bq$-rank$(\Gamma)$ is
defined by the $\bq$-rank of $\mathbf{G}$ where $\mathbf{G}$ is an
algebraic group as in Definition \ref{def:arithmetic}.

A closed subgroup $\mathbf{P} \subset \mathbf{G}$ defined over
$\bq$ is called \emph{rational parabolic subgroup} if $\mathbf{P}$
contains a maximal, connected solvable subgroup of $\mathbf{G}$.
For any rational parabolic subgroup $\mathbf{P}$ of $\mathbf{G}$,
one obtains the \emph{rational Langlands decomposition} of
$P=\mathbf{P}(\br)$:
$$P = N_{\mathbf{P}} \times A_\mathbf{P} \times M_\mathbf{P},$$
where $N_\mathbf{P}$ is the real locus of the unipotent radical
$\mathbf{N}_\mathbf{P}$ of $\mathbf{P}$, $A_\mathbf{P}$ is a
stable lift of the identity component of the real locus of the
maximal $\bq$-torus in the Levi quotient $\mathbf{P} /
\mathbf{N}_\mathbf{P}$ and $M_\mathbf{P}$ is a stable lift of the
real locus of the complement of the maximal $\bq$-torus in
$\mathbf{P} / \mathbf{N}_\mathbf{P}$.

Let $X=G/K$ be the associated symmetric space of noncompact type
with a maximal compact subgroup $K$ of $G$. Write
$X_\mathbf{P}=M_\mathbf{P}/(K \cap M_\mathbf{P})$. Let us denote
by $\tau :  M_\mathbf{P} \ra X_\mathbf{P}$ the canonical
projection. Fix a base point $x_0 \in X$ whose stabilizer group is
$K$. Then we have an analytic diffeomorphism
$$\mu : N_{\mathbf{P}} \times A_\mathbf{P} \times X_\mathbf{P} \ra X, \ (n, a, \tau(m)) \ra nam \cdot x_0,$$
which is called the \emph{rational horocyclic decomposition} of $X$.
For more detail, see \cite[Section III.2]{BJ}.

\subsection{Precise reduction theory}
Let $\mathfrak{g}$ and $\mathfrak{a}_\mathbf{P}$ denote the Lie
algebras of the Lie groups $G$ and $A_\mathbf{P}$ defined above.
Then the adjoint action of $\mathfrak{a}_\mathbf{P}$ on
$\mathfrak{g}$ gives a root space decomposition: $$
\mathfrak{g}=\mathfrak{g}_0 + \sum_{\alpha \in
\Phi(\mathfrak{g},\mathfrak{a}_\mathbf{P})} \mathfrak{g}_\alpha,$$
where $$\mathfrak{g}_\alpha = \{Z \in \mathfrak{g} \ | \
ad(A)(Z)=\alpha(A)Z \text{ for all }A\in \mathfrak{a}_\mathbf{P}
\},$$ and $\Phi(\mathfrak{g},\mathfrak{a}_\mathbf{P})$ consists of
those nontrivial characters $\alpha$ such that
$\mathfrak{g}_\alpha \neq 0$. It is known that
$\Phi(\mathfrak{g},\mathfrak{a}_\mathbf{P})$ is a root system. Fix
an order on $\Phi(\mathfrak{g},\mathfrak{a}_\mathbf{P})$ and
denote by $\Phi^+(\mathfrak{g},\mathfrak{a}_\mathbf{P})$ the
corresponding set of positive roots. Define $$\rho_\mathbf{P} =
\sum_{\alpha \in \Phi^+(\mathfrak{g},\mathfrak{a}_\mathbf{P})}
(\dim \mathfrak{g}_\alpha) \alpha.$$ Let
$\Phi^{++}(\mathfrak{g},\mathfrak{a}_\mathbf{P})$ be the set of
simple positive roots.

Since we consider only $\bq$-rank $1$ arithmetic lattices, we
restrict ourselves from now on to the case
$\bq$-rank$(\mathbf{G})=1$. Then the followings hold:
\begin{itemize}
\item[(1)] All proper rational parabolic subgroups of $\mathbf{G}$ are minimal.
\item[(2)] For any proper rational parabolic subgroup $\mathbf{P}$ of $\mathbf{G}$, $\dim A_\mathbf{P}=1$.
\item[(3)] The set $\Phi^{++}(\mathfrak{g},\mathfrak{a}_\mathbf{P})$ of simple positive $\bq$-roots contains only a single element.
\end{itemize}

For any proper rational parabolic subgroup $\mathbf{P}$ of
$\mathbf{G}$ and any $t>1$, define $$A_{\mathbf{P},t}=\{ a\in
A_\mathbf{P} \ | \ \alpha(a)>t \},$$ where $\alpha$ is the unique
root in $\Phi^{++}(\mathfrak{g},\mathfrak{a}_\mathbf{P})$. For
bounded sets $U\subset N_\mathbf{P}$ and $V\subset X_\mathbf{P}$,
the set $$\mathcal{S}_{\mathbf{P},U,V,t} = U \times
A_{\mathbf{P},t} \times V \subset N_{\mathbf{P}} \times
A_\mathbf{P} \times X_\mathbf{P}$$ is identified with the subset
$\mu(U \times A_{\mathbf{P},t} \times V)$ of $X=G/K$ by the
horospherical decomposition of $X$ and called a \emph{Siegel set}
in $X$ associated with the rational parabolic subgroup
$\mathbf{P}$. Given a $\bq$-rank $1$ lattice $\Gamma$ in $G$, it
is a well-known result due to A. Borel and Harish-Chandra that
there are only finitely many $\Gamma$-conjugacy classes of
rational parabolic subgroups. Recall the precise reduction theory
in $\bq$-rank $1$ case as follows (see \cite[Proposition III.2.21]{BJ}).

\begin{thm}\label{thm:reduction theory}
Let $\Gamma$ be a $\bq$-rank $1$ lattice in $G$. Let $\mathbf{G}$
denote a semisimple algebraic group defined over $\bq$ with
$\bq$-rank$(\mathbf{G})=1$ as in Definition \ref{def:arithmetic}.
Denote by $\mathbf{P}_1,\ldots,\mathbf{P}_s$ representatives of
the $\Gamma$-conjugacy classes of all proper rational parabolic
subgroups of $\mathbf{G}$. Then there exist a bounded set
$\Omega_0$ in $\Gamma\backslash G/K$ and Siegel sets $U_i \times
A_{\mathbf{P}_i,t_i} \times c_i$, $i=1,\ldots,s$, in $X=G/K$ such
that
\begin{itemize}
\item[(1)] each Siegel set $U_i \times A_{\mathbf{P}_i,t_i}
    \times V_i$ is mapped injectively into $\Gamma \backslash
    X$ under the projection $\pi : X \ra \Gamma\backslash X$,
\item[(2)] the image of $U_i \times V_i$ in $(\Gamma \cap P_i)
    \backslash N_{\mathbf{P}_i} \times X_{\mathbf{P}_i}$ is
    compact,
\item[(3)] $\Gamma \backslash X$ admits the following disjoint
    decomposition $$\Gamma \backslash X = \Omega_0 \cup
    \coprod_{i=1}^s \pi(U_i \times A_{\mathbf{P}_i,t_i} \times
    V_i).$$
\end{itemize}
\end{thm}

Geometrically $B_\mathbf{P}(t)=\mu(N_\mathbf{P} \times
A_{\mathbf{P},t} \times X_\mathbf{P})$ is a horoball for any
proper minimal rational parabolic subgroup $\mathbf{P}$ of
$\mathbf{G}$. Hence each $\mu(U_i \times A_{\mathbf{P}_i,t_i}
\times V_i)$ is a fundamental domain of the cusp group $\Gamma_i =
\Gamma \cap \mathbf{P}(\br)$ acting on the horoball
$B_{\mathbf{P}_i}(t_i)$ and each $\mu(U_i \times V_i)$ is a
bounded domain in the horosphere that bounds the horoball
$B_{\mathbf{P}_i}(t_i)$. Furthermore, each set $\pi(U_i \times
A_{\mathbf{P}_i,t_i} \times V_i)$ corresponds to a cusp of the
locally symmetric space $\Gamma \backslash X$. We refer the reader
to \cite{BJ} for more details.

\subsection{Rational horocyclic coordinates}
Let $\mathbf{P}$ be a proper minimal rational parabolic subgroup
of $\mathbf{G}$ with $\bq$-rank$(\mathbf{G})=1$. The pullback
$\mu^*g$ of the metric $g$ on $X$ to $N_{\mathbf{P}} \times
A_\mathbf{P} \times X_\mathbf{P}$ is given by
$$ds^2_{(n,a,\tau(m))}= \sum_{\alpha \in
\Phi^+(\mathfrak{g},\mathfrak{a}_\mathbf{P})} e^{-2\alpha(\log a)}
h_\alpha \oplus da^2 \oplus d(\tau(m))^2,$$ where
$h_\alpha$ is some metric on $\mathfrak{g}_\alpha$ that smoothly
depends on $\tau(m)$ but is independent of $a$. Choosing
orthonormal bases $\{ N_1,\ldots,N_r\}$ of
$\mathfrak{n}_\mathbf{P}$, $\{Z_1,\ldots,Z_l\}$ of some tangent
space $T_{\tau(m)}X_\mathbf{P}$ and $A \in
\mathfrak{a}_\mathbf{P}$ with $\| A\|=1$, one can obtain
\emph{rational horocyclic coordinates} $ \eta : N_{\mathbf{P}}
\times A_\mathbf{P} \times X_\mathbf{P} \ra \mathbb{R}^r \times
\mathbb{R} \times \mathbb{R}^l$ defined by $$\eta\left(\exp
\left(\sum_{i=1}^r x_iN_i\right), \exp (yA), \exp
\left(\sum_{i=1}^lz_iZ_i\right)\right)=(x_1,\ldots,x_r,y,z_1,\ldots,z_l).$$
We abbreviate $(x_1,\ldots,x_r,y,z_1,\ldots,z_l)$ as $(x,y,z)$.
Then the $G$-invariant Riemannian volume form $dvol_X$ on $X\cong
N_{\mathbf{P}} \times A_\mathbf{P} \times X_\mathbf{P}$ with
respect to the rational horocyclic coordinates is given by
$$dvol_X=h(x,z) \exp^{-2\|\rho_\mathbf{P} \| y} dx dy dz,$$ where
$h(x,z)$ is a smooth function that is independent of $y$. See
\cite[Corollary 4.4]{Bo3}.

Note that all proper rational minimal parabolic subgroups are
conjugate under $\mathbf{G}(\bq)$. Hence the respective root
systems are canonically isomorphic \cite{Bo2} and moreover, one
can conclude $\|\rho_\mathbf{P}\|=\|\rho_{\mathbf{P}'}\|$ for any
two proper minimal rational parabolic subgroups
$\mathbf{P},\mathbf{P}'$ of $\bq$-rank $1$ algebraic group
$\mathbf{G}$.

\section{Straight simplices}\label{sec:StraightSimplex}
Let $X$ be a simply connected complete Riemannian manifold with
nonpositive sectional curvature and $\partial_\infty X$ be the
ideal boundary of $X$. For $x_0,\ldots,x_k \in X$, the straight
simplex $str(x_0,\ldots,x_k)$ is defined inductively as follows:
First, $str(x_0)$ is the point $x_0 \in X$, and $str(x_0,x_1)$ is
the unique geodesic arc from $x_1$ to $x_0$. In general,
$str(x_0,\ldots,x_k)$ is the geodesic cone on
$str(x_0,\ldots,x_{k-1})$ with the top point $x_k$. Since there is
the unique geodesic connecting two points in $X$, each ordered
$(k+1)$-tuple $(x_0,\ldots,x_k)$ determines the unique straight
simplex.

If the sectional curvature of $X$ is strictly negative, one can
define the notion of straight simplex in $X\cup \partial_\infty
X$. For any ordered tuple $(u_0,\ldots,u_k) \in X \cup
\partial_\infty X$, the straight simplex $str(u_0,\ldots,u_k)$ is
well defined as above. A straight simplex $str(u_0,\ldots,u_k)$ is
called an \emph{ideal straight simplex} if at least one of
$u_0,\ldots,u_k$ is in $\partial_\infty X$. In general, however,
an ideal straight simplex is not well defined for a simply
connected Riemannian manifold with nonpositive sectional
curvature. For example, let $X$ be a higher rank symmetric space
and consider two points $\theta_1, \theta_2$ in $\partial_\infty
X$ which cannot be connected by any geodesic in $X$. Then, one
cannot define a straight simplex $str(\theta_1,\theta_2)$ with
ideal vertices $\theta_1, \theta_2$. However, in the particular case that
$x_0,\ldots,x_{k-1} \in X$ and $\theta \in \partial_\infty X$, we
can define an ideal straight simplex
$str(x_0,\ldots,x_{k-1},\theta)$ as usual. This is because there
is the unique geodesic from a point in $X$ to a point in
$\partial_\infty X$. Hence, the geodesic cone on
$str(x_0,\ldots,x_{k-1})$ with the top point $\theta$ is well
defined. We only need such kind of ideal straight simplex to construct our invariant in K-theory for a $\bq$-rank $1$ locally symmetric
space.

\begin{setup}\label{setup}
We will stick to the following notations from now on. Let $G$ be a
noncompact, semisimple Lie group with trivial center and no
compact factors and $X=G/K$ be the associated symmetric space with
a maximal compact subgroup $K$ of $G$. Given a $\bq$-rank $1$
arithmetic lattice $\Gamma$ in $G$, we denote by $\mathbf G$ a
$\bq$-rank $1$ semisimple algebraic group defined over $\bq$ as in
Definition \ref{def:arithmetic}. Let $\mathbf P_1,\ldots,\mathbf
P_s$ be the representatives of the $\Gamma$-conjugacy classes of
all proper rational parabolic subgroups of $\mathbf G$. According
to the precise reduction theory, we fix a fundamental domain $F
\subset X$ as in Theorem \ref{thm:reduction theory} as follows:
$$F=\Omega_0 \cup \coprod_{i=1}^s U_i \times A_{\mathbf{P}_i,t_i} \times V_i$$
Each half-geodesic $A_{\mathbf{P}_i,t_i}$ uniquely determines a
point in $\partial_\infty X$, denoted by $c_i$. Write
$\Gamma_i=\Gamma \cap \mathbf P_i(\br)$ and $N=\Gamma \backslash
X$. Note that each $\Gamma_i$ is the stabilizer of $c_i$ in
$\Gamma$. Since $N$ is tame, $N$ is homeomorphic to the interior
of a compact manifold $M$ with boundary. Let
$\partial_1 M, \ldots,\partial_s M$ be the connected components of
the boundary $\partial M$ of $M$. Then there is a one-to-one
correspondence between $\Gamma_1,\ldots,\Gamma_s$ and $\partial_1
M, \ldots,\partial_s M$. Indeed, we can assume that each
$\partial_iM$ is homeomorphic to the quotient space of a
horosphere based at $c_i$ by the action of $\Gamma_i$.
\end{setup}

\begin{lemma}\label{lem:IdealStraight}
For any $c \in \{c_1,\ldots,c_s\}$, the volume of the ideal
straight simplex $str(x_0,\ldots,x_{d-1},c)$ is finite for any
$x_0,\ldots,x_{d-1} \in X$.
\end{lemma}

\begin{proof}
Let $\mathbf{P}$ be the proper minimal rational parabolic subgroup
associated with $c$. Let $\varphi : \Delta^{d-1} \ra X$ be a
parametrization of $str(x_0,\ldots,x_{d-1})$. Choose a coordinate
system $s=(s_1,\ldots,s_{d-1})$ in $\Delta^{d-1}$. In the rational
horocyclic coordinates of $X=N_{\mathbf{P}} \times A_\mathbf{P}
\times X_\mathbf{P} \cong \mathbb{R}^r \times \mathbb{R} \times
\mathbb{R}^l$, we can write $\varphi(s)=(x(s),y(s),z(s))$. Define
a map $\psi : \Delta^{d-1} \times [0,\infty) \ra X$ by $$
\psi(s,t)=(x(s), y(s)+t,z(s)).$$ A line $x(s) \times \mathbb{R}^+
\times z(s)$ is a geodesic representing $c$ for any $s \in
\Delta^{d-1}$. Hence, it is easy to see that $\psi$ is a
parametrization of $str(x_0,\ldots,x_{d-1},c)$.

Denote by $G(w_1,\ldots,w_k)$ the Gram determinant of
$w_1,\ldots,w_k \in \br^d$. It is a standard fact that
$\sqrt{G(w_1,\ldots,w_k)}$ is the $k$-dimensional volume of the
parallelogram with edges $w_1,\ldots,w_k$. We abbreviate
$\frac{\partial \psi}{\partial s_1},\ldots,\frac{\partial
\psi}{\partial s_{d-1}}$ as $\frac{\partial \psi}{\partial s}$.
Then {\setlength\arraycolsep{2pt}
\begin{eqnarray*}
\lefteqn{\psi^* dvol_X (s,t)} \\
&=& h(x(s),z(s)) e^{-2\|\rho_\mathbf{P}\|(y(s)+t)} \sqrt{G \left(\frac{\partial \psi}{\partial s}(s,t), \frac{\partial \psi}{\partial t}(s,t)\right)} \, ds_1 \cdots ds_{d-1} dt\\
&=& h(x(s),z(s)) e^{-2\|\rho_\mathbf{P}\|(y(s)+t)} \sqrt{G \left(\frac{\partial \varphi}{\partial s}(s), \frac{\partial}{\partial y}(\psi(s,t))\right)} \, ds_1 \cdots ds_{d-1} dt\\
&\leq & h(x(s),z(s)) e^{-2\|\rho_\mathbf{P}\|(y(s)+t)} \sqrt{G\left(\frac{\partial \varphi}{\partial s}(s)\right)} \, ds_1 \cdots ds_{d-1} \, dt
\end{eqnarray*}}
The last inequality follows from $\| \frac{\partial}{\partial
y}(\psi(s,t)) \|=1$. Hence, we have {\setlength\arraycolsep{2pt}
\begin{eqnarray*}
\lefteqn{\mathrm{Vol}(str(x_0,\ldots,x_{d-1},c)) } \\
&= & \int_{\Delta^{d-1} \times [0,\infty)}\psi^* dvol_X  \, ds_1 \cdots ds_{d-1} dt \\
& \leq & \int_{\Delta^{d-1}} h(x(s),z(s)) e^{-2\|\rho_\mathbf{P}\|y(s)} \sqrt{G\left(\frac{\partial \varphi}{\partial s}(s)\right)} \, ds_1 \cdots ds_{d-1} \cdot \int_0^\infty e^{-2\|\rho_\mathbf{P}\|t}\, dt \\
&=& \mathrm{Vol}(str(x_0,\ldots,x_{d-1})) \cdot \frac{1}{2\|\rho_\mathbf{P}\|} < \infty.
\end{eqnarray*}}
This completes the proof.
\end{proof}

\section{Cuspidal completion}\label{sec:CuspidalCompletion}

In this section, we will define the notion of disjoint cone for
$M$ and the cuspidal completion of the classifying space
$B\Gamma$, following \cite[Section 4.2.1]{Ku}.

As we mentioned before, one can identify each component
$\partial_i M$ with the quotient $\Gamma_i \backslash H_i$ where
$H_i$ is the horosphere that bounds a horoball
$B_i=N_{\mathbf{P}_i} \times A_{\mathbf{P}_i,t_i} \times
X_{\mathbf{P}_i}$. Note that such $B_i's$ are disjoint in
$\bq$-rank $1$ case. Hence we have a homeomorphism of tuples $$(M,
\partial_1 M, \ldots, \partial_s M)\ra \left( \Gamma \backslash
\left(X-\cup_{i=1}^s \Gamma B_i \right), \Gamma_1 \backslash
H_1,\ldots,\Gamma_s \backslash H_s \right).$$

\subsection{Disjoint cone of topological spaces} For a
topological space $Y$ and subspaces $A_1,\ldots,A_s$ one can
define a disjoint cone $$\mathrm{Dcone}(\cup_{i=1}^s A_i \ra Y)$$
by coning each $A_i$ to a point $c_i$. In other words,
$\mathrm{Dcone}(\cup_{i=1}^s A_i \ra Y)$ is the space obtained by
gluing $Y$ and $\cup_{i=1}^s \mathrm{Cone}(A_i)$ along
$\cup_{i=1}^s A_i$.
\begin{lemma}\label{relativehomology1}
Let $M$ be a compact, connected, smooth, oriented manifold with
boundary $\partial M$. Then there is an isomorphism
$$H_*\left(M,\partial M\right)\cong
H_*\left(\mathrm{Dcone}\left(\cup_{i=1}^s\partial_iM\rightarrow
M\right)\right)$$ in degrees $*\ge 2$. In particular,
$H_d\left(\mathrm{Dcone}\left(\cup_{i=1}^s\partial_iM\rightarrow
M\right),{\br}\right)\cong {\br}$ if $d=dim\left(M\right)\ge
2$.\end{lemma}
\begin{proof}  It is a well-known consequence of Morse theory that $\left(M,\partial M\right)$ is homotopy equivalent to a pair of CW-complexes. Since homology is preserved under homotopy equivalences, we can henceforth assume for the proof that $\left(M,\partial M\right)$ is a pair of CW-complexes. In particular \cite[VII.\ Corollary 1.4]{Bre93} the inclusion $j:\partial M\rightarrow M$ is a cofibration.

Let $K_j$ be the mapping cone of $j$ with vertex $c$. Then \cite[VII.\ Corollary 1.7]{Bre93} implies that we have isomorphisms$$H_*\left(K_j,v\right)\cong H_*\left(M/\partial M\right)\cong H_*\left(M,\partial M\right),$$ thus $H_*\left(K_j\right)\cong H_*\left(M,\partial M\right)$ for $*\ge 1$.

Let $c_i\in \mathrm{Dcone}\left(\cup_{i=1}^s\partial_iM\rightarrow M\right)$ be the vertex of $\mathrm{Cone}\left(\partial_iM \right)$ for each $i=1,\ldots,s$. The projection $\mathrm{Dcone}\left(\cup_{i=1}^s\partial_iM\rightarrow M\right)\rightarrow K_j$ is a cellular map which maps $c_1,\ldots,c_s$ to $c$ and is an isomorphism of cellular chain groups in degree $\ge 1$. Hence it induces an isomorphism of cellular homology in degree $*\ge 2$. It is well known that cellular and singular homology of a pair of CW-complexes agree, thus the claim follows.
\end{proof}
The argument actually provides an isomorphism $$H_*(\mathrm{Dcone}(\cup_{i=1}^s A_i \ra Y))= H_*(Y, \cup_{i=1}^s
A_i)$$ for $*\geq 2$ whenever $Y$ and $A_i$ are CW-complexes.

\subsection{Disjoint cone of simplicial sets} For a simplicial
set $(S,\partial_S)$ and a symbol $c$, \emph{the cone over $S$
with the cone point $c$} is the quasisimplicial set
$\mathrm{Cone}(S)$ whose $k$-simplicies are either $k$-simplices
in $S$ or cones over $(k-1)$-simplices in $S$ with the cone point
$c$. The boundary operator $\partial$ in $\mathrm{Cone}(S)$ is
defined by $\partial \sigma =\partial_S \sigma$ and $$\partial
\mathrm{Cone}(\sigma)=\mathrm{Cone}(\partial_S \sigma)+(-1)^{\dim
(\sigma)+1}\sigma$$ for $\sigma \in S$.

If $\{T_i \ | \ i\in I \}$ is a family of simplicial subsets of
$S$ indexed over a set $I$, then define the quasisimplicial set
$\mathrm{Dcone}(\cup_{i\in I}T_i \ra S)$ as the pushout
$$ \xymatrixcolsep{3pc}\xymatrix{
\dot\bigcup_{i\in I} T_i \ar[r] \ar[d] &
S \ar[d] \\
\dot\bigcup_{i\in I} \mathrm{Cone}(T_i) \ar[r] &
\mathrm{Dcone}\left(\dot\bigcup_{i\in I}T_i \ra S\right)
}$$

Recall that in Section 2.1 we defined the simplicial set $BG$ for a group $G$. Now let us consider $X=G/K$ a symmetric space of noncompact type and $\Gamma\subset G$ a lattice. We define the \emph{cuspidal completion $BG^{comp}$} of $BG$
to be $$\mathrm{Dcone}\left(\dot\bigcup_{c\in \partial_\infty X}
BG \ra BG\right).$$ In addition, define the \emph{cuspidal completion}
$B\Gamma^{comp}$ of $B\Gamma$ to be
$$\mathrm{Dcone}\left( \bigcup_{i=1}^s B\Gamma_i
\ra B\Gamma \right),$$ where $\Gamma_i$ are parabolic groups. More precisely, $B\Gamma^{comp}$ is the quasisimplicial set whose
$k$-simplices $\sigma$ are either of the form
$$\sigma=(\gamma_1,\ldots,\gamma_k)$$ with $\gamma_1,\ldots,\gamma_k
\in \Gamma$ or for some $i\in \{1,\ldots,s\}$ of the form
$$\sigma=(p_1,\ldots,p_{k-1},c_i)$$ with $p_1,\ldots,p_{k-1} \in
\Gamma_i$.
\begin{lemma}\label{relativehomology}
Let $M$ be a compact, connected, smooth, oriented manifold with boundary $\partial M$. Then
there is an isomorphism $$H_*\left(M,\partial M\right)\cong H_*^{simp}\left(\mathrm{Dcone}\left(\oplus_{i=1}^s C_*\left(\partial_iM\right)\rightarrow C_*\left(M\right)\right)\right)$$ in degrees $\ge 2$. In particular $H_d^{simp}\left(\mathrm{Dcone}\left(\oplus_{i=1}^sC_*\left(\partial_iM\right)\rightarrow C_*\left(M\right);{\br}\right)\right)\cong {\br}$ if $d=dim\left(M\right)\ge 2$. \end{lemma}
\begin{proof}
%
For a simplicial set $S$, we denote by $| S|$ the geometric
realisation of $S$. One can think of $C_*(\partial_iM), C_*(M)$ as
simplicial sets. Note that we have a natural isomorphism between
the simplicial homology of the simplicial set and the singular
homology of its geometric realisation. Thus to derive
\hyperref[relativehomology]{Lemma \ref*{relativehomology}} from
\hyperref[relativehomology1]{Lemma \ref*{relativehomology1}} it is
sufficient to provide an isomorphism
$$H_*\left(\mathrm{Dcone}\left(\cup_{i=1}^s\partial_iM\rightarrow
M\right)\right)\cong H_*\left(|
\mathrm{Dcone}\left(\oplus_{i=1}^s
C_*\left(\partial_iM\right)\rightarrow
C_*\left(M\right)\right)|\right).$$

There is a natural Mayer-Vietoris sequence for CW-complexes (see
the remark after \cite[Prop.\ A.5]{Bre93}), hence the canonical
continuous map $$| \mathrm{Dcone}\left(\oplus_{i=1}^s
C_*\left(\partial_iM\right)\rightarrow
C_*\left(M\right)\right)|\rightarrow
\mathrm{Dcone}\left(\cup_{i=1}^s\partial_iM\rightarrow M\right)$$
yields the following commutative diagram:
$$ \begin{xy}
\xymatrix{
\vdots\ar[d]&\vdots\ar[d]\\
H_*\left(| C_*\left(M\right)|\right)\bigoplus\oplus_{i=1}^s H_*\left(| \mathrm{Cone}\left(C_*\left(\partial_i M\right)\right)|\right)\ar[r]\ar[d]&H_*\left(M\right)\bigoplus\oplus_{i=1}^s H_*\left(\mathrm{Cone}\left(\partial_iM\right)\right)\ar[d]\\
H_*\left(| \mathrm{Dcone}\left(\oplus_{i=1}^s C_*\left(\partial_iM\right)\rightarrow C_*\left(M\right)\right)|\right)\ar[r]\ar[d]&H_*\left(\mathrm{Dcone}\left(\cup_{i=1}^s\partial_iM\rightarrow M\right)\right)\ar[d]\\
\oplus_{i=1}^sH_{*-1}\left(| C_*\left(\partial_iM\right)|\right)\ar[d]\ar[r]&\oplus_{i=1}^sH_{*-1}\left(\partial_iM\right)
\ar[d]\\
\vdots&\vdots
}
\end{xy}$$
We note that $| \mathrm{Cone}\left(C_*\left(\partial_i
M\right)\right)|$ and $\mathrm{Cone}\left(\partial_iM\right)$ are
contractible, hence their homology vanishes in degrees $\ge1$.
Moreover $H_*\left(| C_*\left(M\right)|\right)\rightarrow
H_*\left(M\right)$ and $H_*\left(|
C_*\left(\partial_iM\right)|\right)\rightarrow
H_*\left(\partial_iM\right)$ are isomorphisms: this follows from
\cite[Theorem 2.27]{Hat02} together with the fact that
$H_*\left(X\right)$ is by definition the same as
$H_*^{simp}\left(C_*\left(X\right)\right)$ for any topological
space $X$.

Thus the five lemma implies the wanted isomorphism $$H_*\left(|
\mathrm{Dcone}\left(\oplus_{i=1}^s
C_*\left(\partial_iM\right)\rightarrow
C_*\left(M\right)\right)|\right)\rightarrow
H_*\left(\mathrm{Dcone}\left(\cup_{i=1}^s\partial_iM\rightarrow
M\right)\right).$$
\end{proof}

In the sequel we will consider the situation that two points
$x_i,x$ in
$X:=\mathrm{Dcone}\left(\oplus_{i=1}^sC_*\left(\partial_iM\right)\rightarrow
C_*(M)\right)$ are connected by a $1$-simplex $e_i$ with $\partial
e_i=x-x_i$. For a 1-simplex $\sigma$ with both vertices in $x_i$
we can define "conjugation with $e_i$" by
$C_{e_i}(\sigma):=\overline{e_i}*\sigma*e_i$. In particular, for
$x_i\in\partial_iM$ and if
$\pi_1(\partial_iM,x_i)\rightarrow\pi_1(M,x_i)$ is injective, then
$C_{e_i}$ realizes an isomorphism of $\pi_1(\partial_iM,x_i)$ to a
subgroup $\Gamma_i\subset\pi_1(M,x)$. For a $1$-simplex $\sigma$
with $\partial\sigma=c_i-x_i$ we define
$C_{e_i}(\sigma):=\overline{e_i}*\sigma$.

\begin{defi}\label{gammai}
Let $\left(M,\partial M\right)$ be a pair of topological spaces,
 $\partial_1M,\cdots,\partial_sM$ be the path components of
$\partial M$. Denote by $c_i\in
\mathrm{Dcone}\left(\cup_{i=1}^s\partial_i M\rightarrow M\right)$
the vertex of $\mathrm{Cone}\left(\partial_iM\right)$ for
$i=1,\ldots,s$. Let $x\in M, x_1\in\partial_1M,\ldots,
x_s\in\partial_sM$. For $i\in\left\{1,\ldots,s\right\}$ we define
$$\widehat{C}_*^{x_i}\left(\partial_i M\right)\subset \mathrm{Cone}\left(C_*\left(\partial_iM\right)\right)\subset C_*\left(\mathrm{Dcone}\left(\cup_{i=1}^s\partial_i M\rightarrow M\right)\right)$$
to be the subcomplex freely generated by those simplices in $Cone(\partial_i M)$ for which\\
- either all vertices are in $x_i$ ,\\
- or all but the last vertex is in $x_i$ and the last vertex is in $c_i$.

For $i=1,\ldots,s$ fix a path $e_i$ from $x$ to $x_i$ and the
corresponding $C_{e_i}$. Define
$$\widehat{C}_*^x\left(M\right)\subset C_*\left(\mathrm{Dcone}\left(\cup_{i=1}^s\partial_i M\rightarrow M\right)\right)$$
to be the subcomplex freely generated by those simplices $\sigma$ for which\\
- either all vertices are in $x$,\\
- or  for some $i\in\left\{1,\ldots,s\right\}$ there exists a simplex $\sigma^\prime\subset \widehat{C}_*^{x_i}\left(\partial_i M\right)$ such that $C_{e_i}$ maps the $1$-skeleton of $\sigma^\prime$ to the $1$-skeleton of $\sigma$ (up to homotopy fixing the $0$-skeleton).\end{defi}

We remark that in the last case the homotopy classes (rel.\
$\left\{0,1\right\}$) of all edges between all but the last
vertices belong to $\Gamma_i\subset \pi_1\left(M,x\right)$.

For the statement of the following lemma we will denote by
$$j_1:\widehat{C}_*^x\left(M\right)\rightarrow C_*\left(\mathrm{Dcone}\left(\cup_{i=1}^s\partial_i M\rightarrow M\right)\right)$$
and
$$j_2:\mathrm{Dcone}\left(\oplus_{i=1}^sC_*\left(\partial_iM\right)\rightarrow C_*\left(M\right)\right)\rightarrow C_*\left(\mathrm{Dcone}\left(\cup_{i=1}^s\partial_i M\rightarrow M\right)\right)$$
the inclusions.
\begin{lemma}\label{chainhomotopy}Let $M$ be a compact, connected, smooth,
oriented manifold with boundary $\partial M$. Let $x\in M$. Then
there exists a chain map $$F:
C_*(\mathrm{Dcone}\left(\oplus_{i=1}^sC_*\left(\partial_iM\right)\rightarrow
C_*\left(M\right)\right))\rightarrow
\widehat{C}_*^x\left(M\right)$$ such that $ j_1\circ F$ is chain
homotopic to $j_2$.
\end{lemma}
\begin{proof}
To write out the claim of the theorem: we want to show that there
exist sequences of chain maps
$$F_n:C_n\left(\mathrm{Dcone}\left(\oplus_{i=1}^sC_*\left(\partial_iM\right)\rightarrow
C_*\left(M\right)\right)\right)\rightarrow
\widehat{C}_n^x\left(M\right)$$ and of chain homotopies
$$K_n:C_n\left(\mathrm{Dcone}\left(\oplus_{i=1}^sC_*\left(\partial_iM\right)\rightarrow C_*\left(M\right)\right)\right)\rightarrow
C_{n+1}\left(\mathrm{Dcone}\left(\cup_{i=1}^s\partial_i M\rightarrow M\right)\right)$$
for $n=0,1,2,\ldots$ such that
$$\partial K_n\left(\sigma\right)+K_{n-1}\left(\partial\sigma\right)=F_n\left(\sigma\right)-\sigma$$
for all $\sigma\in C_n\left(\mathrm{Dcone}\left(\oplus_{i=1}^sC_*\left(\partial_iM\right)\rightarrow C_*\left(M\right)\right)\right)$.\\

We will use the procedure for dividing $\Delta^n$ into
$\left(n+1\right)$-simplices which is described in \cite[page
112]{Hat02}. So for each $n\in{\mathbb N}$ we let
$v_{n,0},\ldots,v_{n,n}$ and $w_{n,0},\ldots,w_{n,n}$ be the
vertices of $\Delta^n\times\left\{0\right\}$ and
$\Delta^n\times\left\{1\right\}$, respectively, and for $0\le j\le
n$, we denote by $\kappa_{n,j}:\Delta^{n+1}\rightarrow
\Delta^n\times\left[0,1\right]$ the affine
$\left(n+1\right)$-simplex with vertices
$v_0,\ldots,v_j,w_j,\ldots,w_n$. We will inductively prove a
slightly stronger statement as above, namely we will show that for
each $n$-simplex $\sigma$ in
$\mathrm{Dcone}\left(\oplus_{i=1}^sC_*\left(\partial_iM\right)\rightarrow
C_*\left(M\right)\right)$
one can define a continuous map $L_\sigma:\Delta^n\times\left[0,1\right]\rightarrow \mathrm{Dcone}\left(\cup_{i=1}^s\partial_i M\rightarrow M\right)$ such that $K_n\left(\sigma\right)$ is given by $K_n\left(\sigma\right)=\sum_{j=0}^n L_\sigma\circ \kappa_{n,j}$ (and of course that the so defined $K_n$ satisfies the above properties).\\

Let us first consider $n=0$. A $0$-simplex $\sigma$ in
$\mathrm{Dcone}\left(\oplus_{i=1}^sC_*\left(\partial_iM\right)\rightarrow
C_*\left(M\right)\right)$ is either a $0$-simplex in $M$ or a cone
point $c_i$.

If $\sigma=c_i$, then we define $F_0\left(c_i\right)=c_i$ and
$K_0\left(c_i\right)$ is the 1-simplex mapped constantly to $c_i$.

If the $0$-simplex $\sigma$ belongs to $M-\partial M$, then we
define $F_0\left(\sigma\right)=x$ and $K_0\left(\sigma\right)$ is
some (arbitrarily chosen) $1$-simplex in $M\subset
\mathrm{Dcone}\left(\cup_{i=1}^s\partial_iM\rightarrow M\right)$
with $\partial_0K_0\left(\sigma\right)=x$ and
$\partial_1K_0\left(\sigma\right)=\sigma$.

If the $0$-simplex $\sigma$ belongs to $\partial_iM$, then we first choose some 1-simplex $e_\sigma$ in $\partial_iM$ with $\partial_0e_\sigma=x_i, \partial_1e_\sigma=\sigma$. (If $\sigma=x_i$, we let $e_\sigma$ be the constant $1$-simplex.) Recall from Definition 2.3 that we have fixed a path $e_i$ from $x_i$ to $x$ which yields the isomorphism between $\pi_1\left(\partial_iM,x_i\right)$ and $\Gamma_i$ by conjugation. Define then $F_0\left(\sigma\right)=x$ and $K_0\left(\sigma\right)$ is the $1$-simplex obtained as concatenation of $e_\sigma$ and $e_i$. In particular $\partial_0K_0\left(\sigma\right)=x$ and $\partial_1K_0\left(\sigma\right)=\sigma$. \\

Let us now consider $n=1$. A $1$-simplex $\sigma$ in
$\mathrm{Dcone}\left(\oplus_{i=1}^sC_*\left(\partial_iM\right)\rightarrow
C_*\left(M\right)\right)$ is either a $1$-simplex in $M$ or the
cone (with cone point $c_i$) over a $0$-simplex in $\partial_iM$.
We have defined $K_0\left(\partial_1\sigma\right)$ and
$K_0\left(\partial_0\sigma\right)$. Inclusion
$\partial\Delta^1\rightarrow\Delta^1$ is a cofibration, hence we
have a continuous map
$L_\sigma:\Delta^1\times\left[0,1\right]\rightarrow
\mathrm{Dcone}\left(\cup_{i=1}^s\partial_i M\rightarrow M\right)$
such that $L_\sigma\left(x,0\right)=x$ for $x\in\sigma$ and
$L_\sigma\left(\partial_j\sigma,t\right)=K_0\left(\partial_j\sigma\right)\left(t\right)$
for $j=0,1$. Then define
$K_1\left(\sigma\right):=L_\sigma\circ\kappa_{1,0}+L_\sigma\circ\kappa_{1,1}$
and $F_1\left(\sigma\right)$ by
$F_1\left(\sigma\right)\left(x\right):=L_\sigma\left(x,1\right)$
for $x\in\Delta^1$.

It is clear that $\partial K_1\left(\sigma\right)+K_0\left(\partial\sigma\right)=F_1\left(\sigma\right)-\sigma$ and that $F_0\left(\partial\sigma\right)=\partial F_1\left(\sigma\right)$.

We still have to check that $F_1\left(\sigma\right)\in
\widehat{C}_1^x\left(M\right)$. If $\sigma$ is the cone over a
simplex in $\partial_iM$, then $\partial_0\sigma=c_i$ and
$\partial_1\sigma\in\partial_iM\subset M$, hence
$F_0\left(\partial_0\sigma\right)=c_i$ and
$F_0\left(\partial_1\sigma\right)=x$, thus $F_1(\sigma)\in
\widehat{C}_1^x(M)$.
If $\sigma\in C_1\left(M\right)$, then $\partial_jF_1\left(\sigma\right)=F_0\left(\partial_j\sigma\right)=x$ for $j=0,1$, thus $F_1\left(\sigma\right)\in \widehat{C}_1^x\left(M\right)$. Moreover (this will be needed in the next steps) if $\sigma\in C_1\left(\partial_iM\right)$, then the homotopy class (rel.\ $\left\{0,1\right\}$) of $F_1\left(\sigma\right)$ belongs to $\Gamma_i\subset\pi_1\left(M,x\right)$. Indeed, $F_1\left(\sigma\right)$ is in the homotopy class (rel.\ $\left\{0,1\right\}$) of $\overline{K_0\left(\partial_1\sigma\right)}*\sigma*K_0\left(\partial_0\sigma\right)=\overline{e_i}*\overline{e_{\partial_1\sigma}}*\sigma*e_{\partial_0\sigma}*e_i$, where the bar means the $1$-simplex with opposite orientation. Now $\overline{e_{\partial_1\sigma}}*\sigma*e_{\partial_0\sigma}$ represents an element in $\pi_1\left(\partial_iM,x_i\right)$ and by assumption conjugation with $e_i$ provides to isomorphism to $\Gamma_i$, hence $F_1\left(\sigma\right)$ represents an element in $\Gamma_i$. \\

We now proceed to prove the theorem by induction. Assume that
$F_k$ and $K_k$ have been defined for $k\le n$. We will assume as
part of the inductive hypothesis (and prove as part of the
induction claim) that $F_n\left(\sigma\right)$ has all vertices in
$x$ if $\sigma\in C_*\left(M\right)$ and that
$F_n\left(\sigma\right)$ has its last vertex in $c_i$ if $\sigma$
has its last vertex in $c_i$. (This is satisfied for $n\le 1$ by
the above construction.)

Let $\sigma:\Delta^{n+1}\rightarrow
\mathrm{Dcone}\left(\cup_{i=1}^s\partial_i M\rightarrow M\right)$
be an $\left(n+1\right)$-simplex in

$\mathrm{Dcone}\left(\oplus_{i=1}^sC_*\left(\partial_iM\right)\rightarrow
C_*\left(M\right)\right)$. By the inductive hypothesis we have for
$j=0,\ldots,n+1$ a continuous map $L_{\partial_j\sigma}:
\Delta^n\times\left[0,1\right]\rightarrow
\mathrm{Dcone}\left(\cup_{i=1}^s\partial_i M\rightarrow M\right)$
such that $K_n\left(\partial_j\sigma\right)$ is given by
$K_n\left(\partial_j\sigma\right)=\sum_{l=0}^n
L_{\partial_j\sigma}\circ \kappa_{n,l}$. (In particular
$L_{\partial_j\sigma}\left(x,0\right)=x$ for
$x\in\partial_j\Delta^n$.) Since the inclusion
$\partial\Delta^n\rightarrow\Delta^n$ is a cofibration by
\cite[VII. Corollary 1.4]{Bre93} we have a continuous map
$L_\sigma:\Delta^{n+1}\times\left[0,1\right]\rightarrow
\mathrm{Dcone}\left(\cup_{i=1}^s\partial_i M\rightarrow M\right)$
such that $L_\sigma|_{\Delta^{n+1}\times\left\{0\right\}}$ agrees
with $\sigma$ (after the obvious identification of $\Delta^{n+1}$
with $\Delta^{n+1}\times\left\{0\right\}$) and for
$j=0,\ldots,n+1$
$L_\sigma|_{\partial_j\Delta^{n+1}\times\left[0,1\right]}$ agrees
with $L_{\partial_j\sigma}$. Then define
$$K_{n+1}\left(\sigma\right):=\sum_{j=0}^{n+1}L_\sigma\circ\kappa_{n+1,j}$$
and $$F_{n+1}\left(\sigma\right):=L\circ \tau_{n+1},$$ where
$\tau_{n+1}:\Delta^{n+1}\rightarrow
\Delta^{n+1}\times\left[0,1\right]$ is defined by
$\tau_{n+1}\left(x\right)=\left(x,1\right)$.

It is clear by construction that $\partial K_{n+1}\left(\sigma\right)+K_{n}\left(\partial\sigma\right)=F_{n+1}\left(\sigma\right)-\sigma$ and that $\partial F_{n+1}\left(\sigma\right)=F_n\left(\partial\sigma\right)$.

We have to check that $F_{n+1}\left(\sigma\right)\in
\widehat{C}_{n+1}^x\left(M\right)$. If $\sigma$ is an
$\left(n+1\right)$-simplex in $M$, then all $\partial_j\sigma$ are
$n$-simplices in $M$, hence by induction all vertices of all
$F_n\left(\partial_j\right)$ are in $x$. Because of  $\partial_j
F_{n+1}\left(\sigma\right)=F_n\left(\partial_j\sigma\right)$ this
implies that all vertices of $F_{n+1}\left(\sigma\right)$ are in
$x$, hence $F_{n+1}\left(\sigma\right)\in
\widehat{C}_{n+1}^x\left(M\right)$.

If $\sigma$ is the cone (with cone point $c_i$) over an
$n$-simplex $\tau=\partial_n\sigma$, then we have by inductive
hypothesis that $F_n\left(\partial_{n+1}\sigma\right)$ has all its
vertices in $x$ and moreover that all
$F_n\left(\partial_j\sigma\right)$ with $0\le j\le n$ have their
last vertex in $c_i$. Because of  $\partial_j
F_{n+1}\left(\sigma\right)=F_n\left(\partial_j\sigma\right)$ this
implies that $F_{n+1}\left(\sigma\right)$ has its last vertex in
$c_i$ and the remaining vertices in $x$. Moreover, if $n+1=2$,
then $\partial_2\sigma\in C_1\left(\partial_iM\right)$ and it
follows (from the construction for $n=1$) that the homotopy class
(rel.\ $\left\{0,1\right\}$) of
$\partial_2F_2\left(\sigma\right)=F_1\left(\partial_2\sigma\right)$
belongs to $\Gamma_i\subset\pi_1\left(M,x\right)$. If $n+1\ge 3$,
then, since each edge of $\sigma$ is an edge of some
$\partial_j\sigma$ and since $F_n\left(\partial_j\sigma\right)\in
\widehat{C}_n^x\left(M\right)$, it follows that the homotopy
classes (rel.\ $\left\{0,1\right\}$) of all edges between all but
the last vertices belong to
$\Gamma_i\subset\pi_1\left(M,x\right)$. Thus
$F_{n+1}\left(\sigma\right)\in \widehat{C}_{n+1}^x\left(M\right)$.
\end{proof}

\begin{cor}\label{homologyofhatiso} Let $M$ be a compact, connected, smooth, oriented,
aspherical manifold with aspherical $\pi_1$-injective boundary $\partial
M=\partial_1M\cup\ldots\cup\partial_sM$. Let $x\in M$. Assume
$\Gamma_i\cap\Gamma_j=0$ for $i\not=j$, where
$\Gamma_i\subset\pi_1\left(M,x\right)$ for $i=1,\ldots,s$ is
defined by \hyperref[gammai]{Definition \ref*{gammai}}. Then the
chain map $$F:
\mathrm{Dcone}\left(\oplus_{i=1}^sC_*\left(\partial_iM\right)\rightarrow
C_*\left(M\right)\right)\rightarrow
\widehat{C}_*^x\left(M\right)$$ induces an isomorphism of homology
groups.\end{cor}

\begin{proof}
\hyperref[chainhomotopy]{Lemma \ref*{chainhomotopy}} implies that
$j_{1*}F_*=j_{2*}$ and \hyperref[relativehomology]{Lemma
\ref*{relativehomology}} implies that $j_{2*}$ is an isomorphism.
Hence $F_*$ is injective. It remains to show that $F_*$ is
surjective, i.e., that every cycle in
$\widehat{C}_*^x\left(M\right)$ is homologous to some cycle of the
form $F_*z$ with $z$ a cycle in
$\mathrm{Dcone}\left(\oplus_{i=1}^sC_*\left(\partial_iM\right)\rightarrow
C_*\left(M\right)\right)$.

Let $\sum_{j=1}^ra_j\sigma_j\in \widehat{C}_*^x\left(M\right)$ be a cycle. Let $$J^{deg}=\left\{j: \sigma_j\mbox{\ has\ an\ edge\ representing\ }0\in\pi_1\left(M,x\right)\right\}.$$ The same argument as in the proof of \cite[Lemma 5.15]{Ku} shows that $\sum_{j\in J^{deg}}a_j\sigma_j$ is a $0$-homologous cycle, thus $\sum_{j\not\in J^{deg}}a_j\sigma_j$ is homologous to $\sum_{j=1}^ra_j\sigma_j$. We can and will therefore without loss of generality assume that no $\sigma_j$ has an edge representing $0\in\pi_1\left(M,x\right)$.

Let $c_1,\ldots,c_s$ be the cone points and for
$i\in\left\{1,\ldots,s\right\}$ let $$J_i=\left\{j:\sigma_j\mbox{\
has\ its\ last\ vertex\ in\ }c_i\right\}.$$ We note that for
$i\not=l$ a simplex in $J_i$ can not have a face in common with a
simplex in $J_l$. Indeed such a face would have edges representing
elements in $\Gamma_i\subset \pi_1\left(M,x\right)$ and
$\Gamma_l\subset\pi_1\left(M,x\right)$ which is impossible because
of $\Gamma_i\cap\Gamma_l=\emptyset$.

Now let $K$ be the simplicial complex defined as a union
$K=\Delta_1\cup\ldots\cup\Delta_s$ of homeomorphic images of the
$d$-dimensional standard simplex with identifications
$\partial_i\Delta_j=\partial_k\Delta_l$ if and only if
$\partial_i\sigma_j=\partial_k\sigma_l$. Let $\sigma:K\rightarrow
\mathrm{Dcone}\left(\cup_{i=1}^s\partial_iM\rightarrow M\right)$
be defined by $\sigma |_{\Delta_j}=\sigma_j$, where the
homeomorphism from $\Delta_j$ to the standard simplex is
understood. By construction,
$\sum_{j=1}^ra_j\sigma_j=\sigma_*\left[\sum_{j=1}^ra_j\Delta_j\right]$.

We will now homotope $\sigma$
such that its image becomes a chain in

$\mathrm{Dcone}\left(\oplus_{i=1}^sC_*\left(\partial_iM\right)\rightarrow
C_*\left(M\right)\right)$. First, if $j\in J_i$, then we homotope
all but the last vertex of $\Delta_j$ from $x$ to $x_i$ along the
path $e_i$ from \hyperref[gammai]{Definition \ref*{gammai}}. Since
for $i\not=l$ simplices in $J_i$ and $J_l$ have no face in common
this can be done simultaneously for all $\Delta_j$ with $j\in
J_1\cup\ldots\cup J_s$. By successive application of the
cofibration property this homotopy can be extended to all of $K$.
For $j\in J_i$ it follows from the definition of $\Gamma_i$ in
\hyperref[gammai]{Definition \ref*{gammai}} that after this
homotopy the edges of $\Delta_j$ opposite to the cone point are
all mapped to loops at $x_i$ homotopic rel. $\left\{0,1\right\}$
into $\partial_iM$. We may thus (using again the cofibration
property to successively extend the homotopy from the $1$-skeleton
to $K$) further homotope $\sigma$ to have all these edges in
$\partial_iM$, and the remaining edges of $\Delta_i$ mapped to
$\mathrm{Cone}\left(\partial_iM\right)$. Finally, since $M$ and
$\partial_iM$ are aspherical we have $\pi_{*\ge 2}(M,\partial M)=0$, thus we can successively further homotope
$\sigma$ such that:

- for $j\in J_i$ all higher-dimensional
subsimplices and finally $\Delta_j$ are mapped to $\partial_iM$
(if they don't contain the cone point) or to
$\mathrm{Cone}\left(\partial_iM\right)$ (if they do contain the
cone point),

- for $j\not\in J_1\cup\ldots\cup J_s$ all higher-dimensional
subsimplices and finally $\Delta_j$ are mapped to $M$.

Thus we obtain a cycle $c$ in
$\mathrm{Dcone}\left(\oplus_{i=1}^sC_*\left(\partial_iM\right)\rightarrow
C_*\left(M\right)\right)$. By construction $F_*c$ is homotopic,
hence homologous, to $\sum_{j=1}^ra_j\sigma_j$.
\end{proof}

\begin{cor}\label{homologyofhat}Under the assumptions of \hyperref[homologyofhatiso]{Corollary \ref*{homologyofhatiso}} we have $$H_d\left(\widehat{C}_*^x\left(M\right);{\br}\right)\cong{\br}$$
for $d=dim(M)\ge 2$ and $x\in M$.\end{cor}

\section{Eilenberg-MacLane map}
Recall the homeomorphism of tuples as we describe in Section
\ref{sec:CuspidalCompletion}, $$(M, \partial_1 M, \ldots,
\partial_s M)\ra \left( \Gamma \backslash \left(X-\cup_{i=1}^s
\Gamma B_i \right), \Gamma_1 \backslash H_1,\ldots,\Gamma_s
\backslash H_s \right).$$ Let $c_i$ denote the cone point of
$\mathrm{Cone}(\partial_i M)$. Identifying each
$\mathrm{Cone}(\partial_i M)-c_i$ with $\Gamma_i \backslash B_i$,
we have a homeomorphism $$\Gamma \backslash X \ra
\mathrm{Dcone}(\cup_{i=1}^s
\partial_i M \ra M) - \{c_1,\ldots,c_s\}$$ extending the
homeomorphism of tuples above. Composition of the universal covering $X \ra
\Gamma \backslash X$ with this homeomorphism yields a covering map
$$X \ra \mathrm{Dcone}(\cup_{i=1}^s \partial_i M \ra M) -
\{c_1,\ldots,c_s\}.$$ Then we finally have a projection map
$$\pi : X \cup \cup_{i=1}^s \Gamma \partial_\infty B_i \ra \mathrm{Dcone}(\cup_{i=1}^s \partial_i M \ra M)$$
such that $\pi|_X : X \ra \mathrm{Dcone}(\cup_{i=1}^s \partial_i M
\ra M) - \{c_1,\ldots,c_s\}$ is a covering, $\pi|_{\Gamma B_i} :
\Gamma B_i \ra \mathrm{Cone}(\partial_i M)-C_i$ is a covering with
deck group $\Gamma$ and $\pi$ maps $\Gamma \partial_\infty B_i$ to
$c_i$ for $i=1,\ldots,s$. Due to this projection map, we can
define the notion of (ideal) straight simplex in
$\mathrm{Dcone}(\cup_{i=1}^s
\partial_i M \ra M)$.

\begin{defi}\label{def:straight}
We say that a $k$-simplex $\sigma$ in $\mathrm{Dcone}(\cup_{i=1}^s
\partial_i M \ra M)$ is \emph{straight} if $\sigma$ is of the form
$\pi(str(u_0,\ldots,u_k))$ for $u_0,\ldots,u_k \in X \cup
\cup_{i=1}^s \Gamma \partial_\infty B_i$.
\end{defi}

\begin{remark}
In the $\br$-rank $1$ case, every $\partial_\infty B_i$ consists of a
point in $\partial_\infty X$ and for any ordered pair $ (u_0,\ldots,u_k) \in X \cup
\cup_{i=1}^s \Gamma \partial_\infty B_i$, $str(u_0,\ldots,u_k)$ is well
defined. In contrast, in the higher rank
case, each $\partial_\infty B_i$ is not a point. More precisely, if $B_i$ is any horoball centered at $z_i$,
then $$\partial_\infty B_i =\left\{ w \in \partial_\infty X \ \Big| \ Td(z_i,w)\leq \frac{1}{2}\pi \right\},$$
where $Td$ is the Tits metric on $\partial_\infty X$ (see \cite{Hattori}).
Furthermore, $str(u_0,\ldots,u_k)$ may not be defined for some ordered pair $(u_0,\ldots,u_k)$.
\end{remark}

We denote $$\widehat{C}_*(M) :=
C_*(\mathrm{DCone}(\cup_{i=1}^s \partial_i M \ra M)).$$
Recall from \hyperref[gammai]{Definition \ref*{gammai}} that we choose base points $x_0,x_1,\ldots,x_s$ of $M$, $\partial_1 M,\ldots,\partial_s M$ respectively and identify $\pi_1(\partial_i M,x_i)$
with a subgroup $\Gamma_i$ of $\pi_1(M,x_0)$ by choosing a path connecting $x_0$ and
$x_i$ for $i=1,\ldots,s$.


The assumptions of \hyperref[homologyofhat]{Corollary \ref*{homologyofhat}} are satisfied for $\bq$-rank 1 spaces, thus we have
$$H_d(\widehat{C}_*^{x_0}(M),\bq) =
H_d(\widehat{C}_*(M),\bq)=H_d(M,\partial M,\bq)=\bq.$$

We define a chain complex
$$\widehat{C}^{str,x_0}_*(M):=\mathbb{Z}[\{ \sigma \in
\widehat{C}_*^{x_0}(M) \ | \ \sigma \text{ is straight} \}],$$ and
for the $s$-tuple $\left(c_1,\ldots,c_s\right)$ with
$c_i\in\partial_\infty B_i$ (see the setup in Section
\ref{setup}), we define the subcomplex
$\widehat{C}_*^{str,x_0,c}(M)$ of $\widehat{C}_*^{str,x_0}(M)$
freely generated by those simplices that are either of the form
$$\sigma=\pi(str(\tilde{x}_0,\gamma_1\tilde{x}_0,\ldots,\gamma_1\cdots
\gamma_k \tilde{x}_0))$$ where $\tilde{x}_0$ is a lift of $x_0$
and $\gamma_1,\ldots,\gamma_k \in \Gamma$ or of the form
$$\sigma=\pi(str(\tilde{x}_0,p_1\tilde{x}_0,\ldots,p_1\cdots
p_{k-1}\tilde{x}_0,c_i))$$ for $p_1,\ldots,p_{k-1} \in \Gamma_i$
and $i \in \{1,\ldots,s\}$.

\begin{lemma}\label{lem:EMmap}
The following hold.
\begin{itemize}
\item[(a)] There is an isomorphism of chain complexes $$\Phi :
    \widehat{C}^{str,x_0,c}_*(M) \ra
    C^{simp}_*(B\Gamma^{comp}).$$
\item[(b)] The inclusion $\widehat{C}^{str,x_0,c}_*(M) \ra
    \widehat{C}_*(M)$ induces an isomorphism
    $$H_d(\widehat{C}^{str,x_0,c}_*(M),\mathbb{Q}) \ra
    H_d(\mathrm{Dcone}(\cup_{i=1}^s \partial_i M \ra
    M),\mathbb{Q}).$$
\item[(c)] The composition of $\Phi^{-1}$ with the inclusion $
    \widehat{C}^{str,x_0,c}_*(M) \ra \widehat{C}_*(M)$ induces
    an isomorphism $$ \mathrm{EM}_d :
    H^{simp}_d(B\Gamma^{comp},\bq) \ra
    H_d(\mathrm{Dcone}(\cup_{i=1}^s \partial_i M \ra
    M),\bq).$$
\end{itemize}
\end{lemma}

\begin{proof}
(a) We define a chain isomorphism $\Phi :
\widehat{C}^{str,x_0,c}_*(M) \ra C^{simp}_*(B\Gamma^{comp})$ as
follows: For a straight $k$-simplex $\sigma$ of the form
$$\sigma=\pi(str(\tilde{x}_0,\gamma_1\tilde{x}_0,\ldots,\gamma_1\cdots
\gamma_k \tilde{x}_0)),$$ define
$\Phi(\sigma)=(\gamma_1,\ldots,\gamma_k).$ For a straight
$k$-simplex $\sigma$ of the form
$$\sigma=\pi(str(\tilde{x}_0,p_1\tilde{x}_0,\ldots,p_1\cdots p_{k-1}\tilde{x}_0,c_i))$$
with $p_1,\ldots,p_{k-1} \in \Gamma_i$ and $i \in \{1,\ldots,s\}$,
define $\Phi(\sigma)=(p_1,\ldots,p_{k-1},c_i).$ Extending this
linearly to $\widehat{C}^{str,x_0}_*(M)$, we obtain a chain map
$$\Phi : \widehat{C}^{str,x_0,c}_*(M) \ra
C^{simp}_*(B\Gamma^{comp}).$$

Conversely, one can define a chain map $\Psi :
C^{simp}_*(B\Gamma^{comp})\ra \widehat{C}^{str,x_0,c}_*(M)$ as
follows: For a $k$-simplex of the form
$(\gamma_1,\ldots,\gamma_k)$, define
$$\Psi(\gamma_1,\ldots,\gamma_k)=\pi(str(\tilde{x}_0,\gamma_1\tilde{x}_0,\ldots,\gamma_1\cdots
\gamma_k \tilde{x}_0)).$$ For a $k$-simplex of the form
$(p_1,\ldots,p_{k-1},c_i)$ with $p_1,\ldots,p_{k-1}\in \Gamma_i$,
define
$$\Psi(p_1,\ldots,p_{k-1},c_i)=\pi(str(\tilde{x}_0,p_1\tilde{x}_0,\ldots,p_1\cdots
p_{k-1}\tilde{x}_0,c_i)).$$
By construction, it is clear that $\Phi$ and $\Psi$ are inverse to each other.
Hence, $\Phi$ and $\Psi$ are isomorphisms between the two chain complexes.\\

(b) Note that one can extend the domain of $\Phi$ to
$\widehat{C}^{x_0}_*(M)$ in the following way. Let $\sigma$ be a
$k$-simplex in $\widehat{C}^{x_0}_*(M)$ and $e_0,\ldots,e_k$ be
the vertices of the standard simplex $\Delta^k$. For
$i=1,\ldots,k$, let $\zeta_i$ be the standard sub-1-simplex with
$\partial \zeta_i=e_i-e_{i-1}$. Define
$$\Phi(\sigma)=([\sigma|_{\zeta_1}],\ldots,[\sigma|_{\zeta_k}]),$$
where each $[\sigma|_{\zeta_i}]\in \Gamma=\pi_1(M,x_0)$ is the
homotopy class of $\sigma|_{\zeta_i}$ if all vertices of $\sigma$
are in $x_0$ and
$$\Phi(\sigma)=([\sigma|_{\zeta_1}],\ldots,[\sigma|_{\zeta_{k-1}}],c_i),$$
if the last vertex of $\sigma$ is $C_i$.
Then consider the composition of maps as follows:
$$ \xymatrixcolsep{2pc}\xymatrix{
\widehat{C}^{str,x_0,c}_*(M) \ar[r]^-{i} & \widehat{C}^{x_0}_*(M)
\ar[r]^-{\Phi} & C^{simp}_*(B\Gamma^{comp}) \ar[r]^-{\Psi} & \widehat{C}^{str,x_0,c}_*(M)
}$$
Obviously, $\Psi \circ \Phi \circ i = id$ and thus we have an injective homomorphism
$$i_* : H_d(\widehat{C}^{str,x_0,c}_*(M),\bq) \ra  H_d(\widehat{C}^{x_0}_*(M),\bq).$$
From \hyperref[homologyofhat]{Corollary \ref*{homologyofhat}} we obtain $H_d(\widehat{C}^{x_0}_*(M),\bq)\cong\bq$, thus $ H_d(\widehat{C}^{str,x_0,c}_*(M),\bq)$ as a vector space
over $\bq$ is at most one dimensional.
On the other hand \hyperref[volume2]{Lemma \ref*{volume2}} below
shows that $EM_d^{-1}[M,\partial M]$ is a nontrivial element in\\
$H_d(B\Gamma^{comp},\br)\cong
H_d(\widehat{C}^{str,x_0,c}_*(M),\br)$, because evaluation of some
cocycle is not zero. Hence, $i_*$ is actually an isomorphism.
Furthermore, considering that by
\hyperref[homologyofhatiso]{Corollary \ref*{homologyofhatiso}} the
inclusion $\widehat{C}^{x_0}_*(M) \ra \widehat{C}_*(M)$ is a
homology equivalence, one can conclude that the inclusion
$\widehat{C}^{str,x_0,c}_*(M) \ra \widehat{C}_*(M)$ induces an
isomorphism $$H_d(\widehat{C}^{str,x_0,c}_*(M),\mathbb{Q}) \ra
H_d(\mathrm{Dcone}(\cup_{i=1}^s \partial_i M \ra M),\mathbb{Q}).$$

Finally, (c) follows from (a) and (b).
\end{proof}

\begin{remark}
In the $\br$-rank $1$ case, the geodesic straightening map $str_*
: \widehat{C}^{x_0}_*(M) \ra \widehat{C}^{str,x_0}_*(M)$ that is a
left-inverse to the inclusion $\widehat{C}^{str,x_0}_*(M) \subset
\widehat{C}^{x_0}_*(M)$ is well defined and
$C_*^{simp}(B\Gamma^{comp})$ is isomorphic to
$\widehat{C}^{str,x_0}_*(M)=\widehat{C}^{str,x_0,c}_*(M)$. (See
\cite{Ku}). However, in the higher rank case, the straightening
map $str_* : \widehat{C}^{x_0}_*(M) \ra
\widehat{C}^{str,x_0}_*(M)$ is not well defined as we mentioned in
Section \ref{sec:StraightSimplex}. Hence
$C_*^{simp}(B\Gamma^{comp})$ is not isomorphic to
$\widehat{C}^{str,x_0}_*(M)$ but is isomorphic to its subcomplex
$\widehat{C}^{str,x_0,c}_*(M)$. Despite such differences with
$\br$-rank $1$ case, we obtain the homology class
$$EM_d^{-1}[M,\partial M] \in H^{simp}_d(B\Gamma^{comp},\bq)$$ in
the same manner as in the $\br$-rank $1$ case. This will enable us in Section 7 to define
an invariant in K-theory for a $\bq$-rank $1$ locally symmetric
space.
\end{remark}

Recall that the volume cocycle $comp(\nu_d)=c\nu_d \in C^d_{simp}(BG)$
is defined by
$$c\nu_d(g_1,\ldots,g_d)=\int_{str(\tilde{x}_0,g_1 \tilde{x}_0, \ldots, g_1\cdots g_d \tilde{x}_0)} dvol_X,$$
where $dvol_X$ is the $G$-invariant Riemannian volume form on $X=G/K$.
If $X$ is a $\br$-rank $1$ symmetric space, one can extend this
volume cocycle $c\nu_d$ to a cocycle $\overline{c\nu}_d \in
C^d_{simp}(BG^{comp})$. In the higher rank case, one cannot
obtain an extended volume cocycle in $C^d_{simp}(BG^{comp})$.
However, we can extend the volume cocycle in $C^d_{simp}(B\Gamma)$
to at least a cocycle in $C^d_{simp}(B\Gamma^{comp})$. Define a
cocycle $\overline{c\nu}_d \in C^d_{simp}(B\Gamma^{comp})$ as
follows: For a $d$-simplex $(\gamma_1,\ldots,\gamma_d)$ with
$\gamma_1,\ldots,\gamma_d \in \Gamma$, define
$$\overline{c\nu}_d(\gamma_1,\ldots,\gamma_d)=c\nu_d(\gamma_1,\ldots,\gamma_d).$$
For a $d$-simplex $(p_1,\ldots,p_{d-1},c_i)$ with
$p_1,\ldots,p_{d-1} \in \Gamma_i$ and $i \in \{1,\ldots,s\}$,
define
$$\overline{c\nu}_d(p_1,\ldots,p_{d-1},c_i)=\int_{str(\tilde{x}_0,p_1
\tilde{x}_0, \ldots, p_1\cdots p_{d-1} \tilde{x}_0,c_i)} dvol_X.$$
It follows from Lemma \ref{lem:IdealStraight} that
$\overline{c\nu}_d$ is well defined. An application of Stokes' Theorem (to compact submanifolds with boundary of every ideal simplex exactly as in the proof of \hyperref[volume2]{Lemma \ref*{volume2}} below) shows that
$\overline{c\nu}_d$ is a cocycle in $C^d_{simp}(B\Gamma^{comp})$.
Hence, the cocycle $\overline{c\nu}_d$ determines a cohomology
class in $H^d_{simp}(B\Gamma^{comp})$, denoted by
$\overline{v}_d$.

\begin{lemma}\label{volume2}
If $N$ is a $\bq$-rank $1$ locally symmetric space of dimension at least $3$, then
$$\langle \overline{v}_d, EM_d^{-1}[M,\partial M] \rangle = \mathrm{Vol}(N).$$
\end{lemma}

\begin{proof}
Let $z$ be a relative fundamental cycle in $C_d(M,\partial M)$
representing the relative fundamental class $[M,\partial M]$. We
think of $M$ as a submanifold of $N$ via the homeomorphism of
tuples $$(M, \partial_1 M, \ldots, \partial_s M)\ra \left( \Gamma
\backslash \left(X-\cup_{i=1}^s \Gamma B_i \right), \Gamma_1
\backslash H_1,\ldots,\Gamma_s \backslash H_s \right)$$ and $z$ as
a chain in $C_d(N)$.

Since $\partial z$ represents $[\partial M]$, we can write
$\partial z =\partial_1 z + \cdots + \partial_s z$ where
$\partial_i z$ is a cycle representing $[\partial_i M]$ for
$i=1,\ldots,s$. Make $z$ a chain $z^0$ in $\widehat{C}^{x_0}_d(M)$
via the chain homotopy $\widehat{C}_*(M) \ra
\widehat{C}^{x_0}_*(M)$ in the proof of \cite[Lemma 8]{Ku}. Note
that $z^0$ is obtained by adding several $1$-dimensional paths to $z$ and hence
$$algvol(z^0)=algvol(z)=\mathrm{Vol}(M).$$

Now, consider a geodesic cone $\mathrm{Cone}_g(\partial_i z)$ over
$\partial_i z$ with the top point $c_i$ for $i=1,\ldots,s$. Due to
$\dim A_{\mathbf{P}_i}=1$, it is not difficult to see that
$$algvol(\mathrm{Cone}_g(\partial_i
z))=(-1)^{d+1}\mathrm{Vol}(\Gamma_i \backslash B_i).$$ Since
$\mathrm{Cone}_g(\partial_i z^0)$ is obtained by adding two
dimensional objects of $N$ to $\mathrm{Cone}_g(\partial_i z)$, its
algebraic volume is not changed, that is,
$$algvol(\mathrm{Cone}_g(\partial_i z^0))=algvol(\mathrm{Cone}_g(\partial_i z))=(-1)^{d+1}\mathrm{Vol}(\Gamma_i \backslash B_i).$$

Now, define $c(z^0)=z^0+(-1)^{d+1}\mathrm{Cone}_g(\partial z^0)$.
Then it can be checked that $c(z^0)$ is a cycle in $\widehat{C}_d^{x_0}(M)$ by
{\setlength\arraycolsep{2pt}
\begin{eqnarray*}
\partial c(z^0) &=& \partial z^0 + (-1)^{d+1}\partial \mathrm{Cone}_g(\partial z^0) \\
&=&\partial z^0 + (-1)^{d+1} \mathrm{Cone}_g(\partial\partial z^0) + (-1)^{d+1} (-1)^d \partial z^0 = 0.
\end{eqnarray*}}
Furthermore, we have
{\setlength\arraycolsep{2pt}
\begin{eqnarray*}
algvol(c(z^0)) &=& \mathrm{Vol}(M)+\sum_{i=1}^s \mathrm{Vol}(\Gamma_i \backslash B_i) \\
&=& \mathrm{Vol}(M)+\mathrm{Vol}(N-M)=\mathrm{Vol}(N).
\end{eqnarray*}}
Note that we can straighten $c(z^0)$ because all vertices of every
ideal simplex in $c(z^0)$ are in $x_0$ except for the last vertex
with $c_i$ for some $i \in \{1,\ldots,s\}$. Furthermore,
$str(c(z^0))$ is in $\widehat{C}^{str,x_0,c}_d(M)$ and represents
$\Psi \circ EM^{-1}_d[M,\partial M]$. To prove the lemma, it is
sufficient to show that
$$algvol(str(c(z^0))) = \mathrm{Vol}(N).$$

Let $H_0, H_1:K\rightarrow
\mathrm{Dcone}(\cup_{i=1}^s\partial_iM\rightarrow M)$ be the
simplicial maps realising the cycles
$c(z^0)+\mathrm{Cone}_g(\partial z^0)$ and $str(c(z^0)+\mathrm{Cone}_g(\partial z^0))$, respectively. (See the construction in the
proof of \ref{homologyofhatiso}.) Let
$H:K\times\left[0,1\right]\rightarrow
\mathrm{Dcone}(\cup_{i=1}^s\partial_iM\rightarrow M)$ be the
straight line homotopy between $c(z^0)+\mathrm{Cone}_g(\partial
z^0)$ and $str(c(z^0)+\mathrm{Cone}_g(\partial z^0))$, such that
$H_0=H(.,0), H_1=H(.,1)$.

The homotopy $H$ yields a chain homotopy $L_* :
\widehat{C}^{str,x_0,c}_*(M)\rightarrow
\widehat{C}^{str,x_0,c}_{*+1}(M)$ from the straightening map $str$
to the identity. (See the construction in the proof of \ref{chainhomotopy}.) This satisfies $\partial L_k + L_{k-1} \partial =
str-id$. Then {\setlength\arraycolsep{2pt}
\begin{eqnarray*}
\lefteqn{str(c(z^0))-c(z^0)} \\
&= & str(z^0)-z^0+(-1)^{d+1}(\mathrm{Cone}_g(str(\partial z^0))-\mathrm{Cone}_g(\partial z^0)) \\
&=& \partial L_d(z^0)+L_{d-1}(\partial z^0) +(-1)^{d+1}\mathrm{Cone}_g(\partial L_{d-1}(\partial z^0) + L_{d-2}(\partial \partial z^0)) \\
&=& \partial L_d(z^0)+L_{d-1}(\partial z^0) +(-1)^{d+1}\partial \mathrm{Cone}_g(L_{d-1}(\partial z^0)) - L_{d-1}(\partial z^0) \\
&=& \partial (L_d(z^0)+(-1)^{d+1}\mathrm{Cone}_g(L_{d-1}(\partial z^0)))
\end{eqnarray*}}
In order to conclude $algvol(str(c(z^0))) = algvol(c(z^0))$ we want to apply Stokes' Theorem to show that the integral of the volume form over $\partial (L_d(z^0)+(-1)^{d+1}\mathrm{Cone}_g(L_{d-1}(\partial z^0)))$ vanishes. It is clear that the integral of the closed form $dvol$ over $\partial L_d(z^0)$ vanishes and thus it remains to look at simplices in $\partial \mathrm{Cone}_g(L_{d-1}(\partial z^0))$.

Since the volume form is defined on the complement of the cone points we can apply Stokes' Theorem to compact submanifolds with boundary of every (ideal) simplex in
Let ${\mathcal{H}}_{ik}$ be a sequence of horospheres converging
towards $c_i$ for $k\rightarrow\infty$. In the following we will
call simplices in $\mathrm{Cone}_g(\partial_iz)$ {\em proper
ideal} simplices if they have a vertex in $c_i$, i.e.\ if they are
not contained in $\partial_iz$. For a proper ideal simplex in
$\mathrm{Cone}_g(\partial_i z)$  its edges are either edges of a
simplex in $\partial_iz$ or otherwise they are geodesics ending in
$c_i$, which therefore are transverse to the horospheres
${\mathcal{H}}_{ik}$. Moreover all higher-dimensional proper ideal
simplices in $\mathrm{Cone}_g(\partial_i z^0)$ and their
straightenings are a union of geodesic lines ending in $c_i$. In
particular all proper ideal simplices occuring in
$\mathrm{Cone}_g(\partial_iz^0)$ and
$str(\mathrm{Cone}_g(\partial_iz^0))$ are transverse to the
${\mathcal{H}}_{ik}$'s.

The Relative Transversality Theorem (see \cite{gp}) yields that any map $H:K\times\left[0,1\right]\rightarrow \mathrm{Dcone}(\cup_{i=1}^s\partial_iM\rightarrow M)$ whose restriction to $K\times\left\{0,1\right\}$ is transverse to $\bigcup_{i.k}{\mathcal{H}}_{ik}$ can be homotoped (by an arbitrarily small homotopy, keeping $K\times\left\{0,1\right\}$ fixed) to a map which is transverse on all of $K\times\left[0,1\right]$. The homotopy can be chosen to fix subsimplices which are already transverse. Thus we can homotope (keeping $c(z^0)$ and $str(c(z^0))$ as well as $L_d(z^0)$ fixed) the simplices in $\mathrm{Cone}_g(L_{d-1}(\partial z^0))$
to be transverse to the ${\mathcal{H}}_{ik}$'s.

Then, for any $\left(d+1\right)$-dimensional simplex
$$\kappa:\Delta^{d+1}\rightarrow
\mathrm{Dcone}\left(\cup_{i=1}^s\partial_iM\rightarrow
M\right)\cong N\cup\left\{cusps\right\}$$ occuring in
$\mathrm{Cone}_g(L_{d-1}(\partial z^0))$, and for each
$k\in{\Bbb N}$, we conclude from transversality that
$\kappa^{-1}\left({\mathcal{H}}_{ik}\right)$ is a $d$-dimensional
submanifold $K_{ik}$ bounding a $\left(d+1\right)$-dimensional
submanifold $\Omega_{ik}\subset\Delta^{d+1}$ which does not
contain the preimage of $c_i$. Thus we can apply Stokes' Theorem
to $\Omega_{ik}$ and obtain
$$\int_{\partial\Delta^{d+1}\cap\Omega_{ik}}H^*dvol_X+\int_{K_{ik}}H^*dvol_X=\int_{\Omega_{ik}}dH^*dvol_X=0$$
because $H^*dvol_X$ is a closed form. We note that $H$ maps the
$d$-dimensional submanifold $K_{ik}$ to the
$\left(d-1\right)$-dimensional submanifold
${\mathcal{H}}_{ik}\subset N$ and therefore $
\int_{K_{ik}}H^*dvol_X=0$. Thus
$$\int_{\partial\Delta^{d+1}\cap\Omega_{ik}}H^*dvol_X=0$$ for all
$k\in{\Bbb N}$. Then, since the countable union $\cup_{k\in{\Bbb
N}}\partial\Delta^{d+1}\cap\Omega_{ik}$ equals
$\Delta^{d+1}\setminus \kappa^{-1}\left(c_i\right)$ and since
$dvol_X$ is defined to be zero on $c_i$, we conclude that
$$\int_{\partial\Delta^{d+1}}H^*dvol_X=0.$$ Summing up over all
$\kappa:\Delta^{d+1}\rightarrow
\mathrm{Dcone}\left(\cup_{i=1}^s\partial_iM\rightarrow M\right)$
occuring in $\mathrm{Cone}_g(L_{d-1}(\partial z^0))$ we
obtain
$$\int_{\partial \mathrm{Cone}_g(L_{d-1}\left(\partial z^0\right))}dvol_X=0.$$

Then,
we have
$$algvol(str(c(z^0)))=\int_{str(c(z^0))}dvol_X=\int_{c(z^0)}dvol_X= \mathrm{Vol}(N),$$
which completes the proof.
\end{proof}

Assume that $d$ is odd. Let $b_d$ be the Borel class in
$H^d_c(\mathrm{SL}(n,\mathbb{C}),\br)$. (It is obtained by restriction of the Borel class $b_d\in H^d_c(\mathrm{GL}(n,\mathbb{C}),\br)$ defined  in Section 2.2.)  According to the Van Est
isomorphism, a representative $\beta_d$ of $b_d$ is given by
$$\beta_d(g_0,\ldots,g_d)=\int_{str(g_0 \tilde{o},\ldots, g_d
\tilde{o})} dbol$$ where $dbol$ is an
$\mathrm{SL}(n,\mathbb{C})$-invariant differential $d$-form on
$\mathrm{SL}(n,\mathbb{C})/\mathrm{SU}(n)$ and $\tilde{o}$ is a
point in $\mathrm{SL}(n,\mathbb{C})/\mathrm{SU}(n)$. Recall that a comparison map $$comp :
C^*_c(\mathrm{SL}(n,\mathbb{C}),\br) \ra
C^*_{simp}(B\mathrm{SL}(n,\mathbb{C}),\br)$$ is defined by
$comp(f)(g_1,\ldots,g_k)=f(1, g_1, g_1 g_2,\ldots,g_1\cdots g_k)$.

As in \cite[Section 4.2.3]{Ku} define $B\mathrm{SL}(n,\mathbb{C})^\text{fb}_d$ as the set
consisting of $d$-simplices in
$B\mathrm{SL}(n,\mathbb{C})^{comp}_d$ which are either
$d$-simplices in $B\mathrm{SL}(n,\mathbb{C})_d$ or of the form
$(p_1,\ldots,p_{d-1},c) $ satisfying
$$\left|\int_{str(\tilde{o},p_1\tilde{o},\ldots, p_1\cdots
p_{d-1}\tilde{o},c)} dbol \right| < \infty.$$ We define
$B\mathrm{SL}(n,\mathbb{C})^\text{fb}$ to be the quasisimplicial
set generated by $B\mathrm{SL}(n,\mathbb{C})^\text{fb}_d$ under
face maps. Then consider a cocycle $\overline{c\beta}_d :
C^{simp}_d(B\mathrm{SL}(n,\mathbb{C})^\text{fb},\br)\ra \br$
defined by
$$\overline{c\beta}_d(g_1,\ldots,g_d)=\int_{str(\tilde{o},g_1 \tilde{o},\ldots, g_1\cdots g_{d-1} \tilde{o})}dbol$$
for $(g_1,\ldots,g_d)\in B\mathrm{SL}(n,\mathbb{C})^\text{fb}_d$,
and
$$\overline{c\beta}_d(p_1,\ldots,p_{d-1},c)=\int_{str( \tilde{o},p_1\tilde{o},\ldots, p_1\cdots p_{d-1}\tilde{o},c)}dbol$$
for $(p_1,\ldots,p_{d-1},c)\in
B\mathrm{SL}(n,\mathbb{C})^\text{fb}_d$. By the construction of
$\overline{c\beta}_d $, it is obvious that $(B_i)^*
(\overline{c\beta}_d)=comp(\beta_d)$ and hence, $(B_i)^*
(\overline{c\beta}_d)$ is a cocycle representing $comp(b_d)$ where
$B_i : B\mathrm{SL}(n,\mathbb{C}) \ra
B\mathrm{SL}(n,\mathbb{C})^\text{fb}$ is the natural inclusion
map.

\begin{lemma}\label{borelclass}
Let $\rho : (G,K) \ra (\mathrm{SL}(n,\mathbb{C}),\mathrm{SU}(n))$
be a representation. Let $j:\Gamma \ra G$ be the natural inclusion
map. Then we have
$$(B_\rho \circ B_j)_* (C^{simp}_d(B\Gamma^{comp})) \subset C^{simp}_d(B\mathrm{SL}(n,\mathbb{C})^\mathrm{fb}),$$
where $B_\rho : BG^{comp} \ra B\mathrm{SL}(n,\mathbb{C})^{comp}$
and $B_j : B\Gamma^{comp} \ra BG^{comp}$ are the induced maps from
$\rho$ and $j$ respectively. Furthermore, there exists a constant
$c_\rho \in \br$ such that $(B_\rho \circ B_j)^*
(\overline{c\beta}_d)$ represents $c_\rho \overline{v}_d$.
\end{lemma}

\begin{proof}
To prove the first statement, it suffices to show that for a
$d$-simplex of the form $(p_1,\ldots,p_{d-1},c_i)$ with
$p_1,\ldots,p_{d-1}\in \Gamma_i$,
$$(B_\rho \circ B_j)_* (p_1,\ldots,p_{d-1},c_i) \subset C^{simp}_d(B\mathrm{SL}(n,\mathbb{C})^\text{fb}).$$

The homomorphism $\rho$ induces a $\rho$-equivariant map $S_\rho :
X \ra Y$ and $\rho_\infty : \partial_\infty X \ra \partial_\infty
Y$ where $X=G/K$ and $Y=\mathrm{SL}(n,\mathbb{C})/\mathrm{SU}(n)$.
Since $S_\rho^*dbol$ is a $G$-invariant $d$-form on $X$ and the
space of $G$-invariant differential $d$-forms on $X$ is generated
by the $G$-invariant Riemannian volume form $dvol_X$, there is a
constant $c_\rho \in \br$ such that $S_\rho^*dbol=c_\rho dvol_X$.
Noting that $S_\rho$ maps geodesics to geodesics, we have
\begin{eqnarray*}
\lefteqn{S_\rho(str(\tilde{x}_0, p_1\tilde{x}_0,\ldots, p_1\cdots p_{d-1}\tilde{x}_0,c_i))} \\
&=&  str(\tilde{o}, \rho(p_1)\tilde{o},\ldots, \rho(p_1)\cdots \rho(p_{d-1})\tilde{o},\rho_\infty (c_i))
\end{eqnarray*}
where $\tilde{x}_0$, $\tilde{o}$ are points in $X$ and $Y$ whose
stabilizers are $K$ and $\mathrm{SU}(n)$ respectively. For
$(g_1,\ldots,g_{k-1},c) \in BG^{comp}$, observe that the induced
map $B_\rho : BG^{comp} \ra B\mathrm{SL}(n,\mathbb{C})^{comp}$ is
defined by
$$B_\rho(g_1,\ldots,g_{k-1},c)=(\rho(g_1),\ldots,\rho(g_{k-1}),\rho_\infty
(c)).$$

From the above observations, we have
{\setlength\arraycolsep{2pt}
\begin{eqnarray*}
\int_{str(\tilde{o}, \rho(p_1)\tilde{o},\ldots, \rho(p_1)\cdots \rho(p_{d-1})\tilde{o},\rho_\infty (c_i))}dbol
&=& \int_{str(\tilde{x}_0, p_1\tilde{x}_0,\ldots, p_1\cdots p_{d-1}\tilde{x}_0,c_i)} (S_\rho)^*dbol \\
&=& \int_{str(\tilde{x}_0, p_1\tilde{x}_0,\ldots, p_1\cdots p_{d-1}\tilde{x}_0,c_i)} c_\rho dvol_X \\
&<& \infty.
\end{eqnarray*}}
Furthermore, the above equation implies
$$(B_\rho \circ B_j)^* (\overline{c\beta}_d)(p_1,\ldots,p_{d-1},c_i)=c_\rho \overline{c\nu}_d (p_1,\ldots,p_{d-1},c_i).$$
Therefore, this implies the second statement.
\end{proof}

\section{Invariants in group homology and K-theory}

In this section, we construct $\gamma(N)\in K_d(\overline{\bq})
\otimes \bq$ for a $\bq$-rank $1$ locally symmetric space
$N=\Gamma \backslash G/K$.

\begin{prop}\label{prop:preimage}
Let $\Gamma \subset \mathbf{G}(\overline{\bq})$ be a $\bq$-rank $1$ lattice. Let $N=\Gamma\backslash G/K$ be the interior of the manifold with boundary $M$. Let $\rho
: \mathbf{G}(\overline{\bq}) \ra \mathrm{SL}(n,\overline{\bq})$ be
a representation. Suppose that $\rho(\Gamma_i)$ is unipotent for
all $i\in \{1,\ldots,s\}$. Then $$ (B_\rho \circ B_j)_*
EM_d^{-1}[M,\partial M] \in
H^{simp}_d(B\mathrm{SL}(n,\overline{\bq}))^\mathrm{fb},\bq)$$ has
a preimage $$\overline{\gamma}(N) \in
H^{simp}_d(B\mathrm{SL}(n,\overline{\bq}),\bq),$$ where $j :
\Gamma \ra \mathbf{G}(\overline{\bq})$ is the natural inclusion
map.
\end{prop}

\begin{proof}
The exact same argument in the proof of \cite[Proposition 1]{Ku}
works under the assumption that every $\rho(\Gamma_i)$ is
unipotent. For the detailed proof, see \cite[Proposition 1]{Ku}.
\end{proof}

Let $A$ be a subring with unit of the ring of complex numbers. One
can regard an element in $H^{simp}_*(B\mathrm{SL}(A),\bq)$ as an
element in $H_*(|B\mathrm{SL}(A)|^+,\bq)$ due to the canonical
identifications $$H^{simp}_*(B\mathrm{SL}(A),\bq) \cong
H_*(|B\mathrm{SL}(A)|,\bq) \cong H_*(|B\mathrm{SL}(A)|^+,\bq).$$
By the Milnor-Moore theorem, the Hurewicz homomorphism
$$K_*(A)=\pi_*(|B\mathrm{GL}(A)|^+)\ra H_*(|B\mathrm{GL}(A)|^+,\bz)$$ gives,
after tensoring with $\bq$, an injective homomorphism
$$I_* : K_*(A)\otimes \bq=\pi_*(|B\mathrm{GL}(A)|^+)\otimes \bq\ra
H_*(|B\mathrm{GL}(A)|^+,\bq).$$ Its image consists of primitive
elements, denoted by $PH_*(|B\mathrm{GL}(A)|^+,\bq)$.

By Quillen, inclusion $|B\mathrm{GL}(A)| \subset
|B\mathrm{GL}(A)|^+$ induces an isomorphism
$$Q_* : PH_*(|B\mathrm{GL}(A)|,\bq) \ra PH_*(|B\mathrm{GL}(A)|^+,\bq) \cong K_*(A) \otimes \bq.$$
Once the projection
$$pr_* : H_*^{simp}(B\mathrm{GL}(A),\bq)\ra PH_*^{simp}(B\mathrm{GL}(A),\bq) \cong PH_*(|B\mathrm{GL}(A)|,\bq)$$ is
fixed, we can define an element $I_*^{-1} \circ Q_* \circ pr_*
(\alpha) \in K_*(A)\otimes \bq$ for each $\alpha \in
H_*^{simp}(B\mathrm{GL}(A),\bq)$. For $h \in K_{2m-1}(A)\otimes
\bq$, define
$$\langle b_{2m-1}, h \rangle= \langle comp(b_{2m-1}),
Q^{-1}_{2m-1} \circ I_{2m-1}(h) \rangle.$$ In particular, when
$A=\overline{\bq}$, note that by \cite[Corollary 2]{Ku} the projection $pr_*$ can be chosen such that $$\langle comp(b_{2m-1}),
pr_{2m-1}(\alpha)\rangle=\langle comp(b_{2m-1}),\alpha \rangle$$
for all $m \in \mathbb{N}$. We refer the reader to \cite[Section
2.5]{Ku} for more details about this.

\begin{thm}\label{thm:Ktheory}
Let $G/K$ be a symmetric space of noncompact type with odd
dimension $d$ and $N=\Gamma \backslash G/K$ be a $\bq$-rank $1$
locally symmetric space. Let $\rho : G \ra
\mathrm{GL}(n,\mathbb{C})$ be a representation with $\rho^*_c b_d
\neq 0$. Suppose that $\rho(\Gamma_i)$ is unipotent for all $i\in
\{1,\ldots,s\}$. Then there is an element $$\gamma(N) \in
K_d(\overline{\bq}) \otimes \bq$$ such that the application of the
Borel class $b_d$ yields
$$\langle b_d, \gamma(N) \rangle = c_\rho \mathrm{Vol}(N)$$
for some constant $c_\rho \neq 0$.
\end{thm}

\begin{proof}
First, note that connected,
semisimple Lie groups are perfect, hence each representation $\rho:G \ra \mathrm{GL}(n,\mathbb{C})$ has image in
$\mathrm{SL}(n,\mathbb{C})$. In addition, we can assume that $\rho$ maps $K$
to $\mathrm{SU}(n)$, which can be achieved upon conjugation.

Once $G$ is identified with $\mathbf G (\br)^0$ via an
isomorphism, by Weil's rigidity theorem, there is a $g\in G$ such
that $g\Gamma g^{-1} \subset \mathbf{G}(\overline{\bq})^0$.
Moreover, it follows from the classification of irreducible
representations of Lie groups \cite{FH} that each representation
$\rho : \mathbf G(\br)^0 \rightarrow \mathrm{GL}(n,\mathbb{C})$ is
isomorphic to a representation $\rho' : \mathbf G(\br)^0
\rightarrow \mathrm{GL}(n,\mathbb{C})$ such that
$\mathbf{G}(\overline{\bq})^0$ is mapped to
$\mathrm{GL}(n,\overline{\bq})$. Thus, we can assume that $\Gamma
\subset \mathbf G (\overline \bq)^0$ and $\rho|_{\mathbf G
(\overline \bq)^0} : \mathbf G (\overline \bq)^0 \ra
\mathrm{SL}(n,\overline \bq)$ is a representation for which every
$\rho(\Gamma_i)$ is unipotent.

According to Proposition \ref{prop:preimage}, we obtain
$\overline{\gamma}(N)\in
H^{simp}_d(B\mathrm{SL}(n,\overline{\bq}),\bq)$. Let us denote by
$\overline{\overline{\gamma}}(N)$ the image of
$\overline{\gamma}(N)$ in
$H^{simp}_d(B\mathrm{GL}(n,\overline{\bq}),\bq)$. Now, we define
$$\gamma(N)=I_d^{-1} \circ Q_d \circ pr_d(\overline{\overline{\gamma}}(N)) \in K_d(\overline{\bq})\otimes \bq.$$
Then, we have
{\setlength\arraycolsep{2pt}
\begin{eqnarray*}
\langle b_d, \gamma(N) \rangle &=& \langle comp(b_d), pr_d(\overline{\overline{\gamma}}(N)) \rangle \\
&=& \langle comp(b_d), \overline{\overline{\gamma}}(N) \rangle = \langle [comp(\beta_d)], \overline{\gamma}(N) \rangle \\
&=& \langle (B_i)^*[\overline{c\beta}_d], \overline{\gamma}(N) \rangle =\langle  [\overline{c\beta}_d], (B_i)_*\overline{\gamma}(N) \rangle \\
&=& \langle [\overline{c\beta}_d], (B_\rho \circ B_j)_*EM_d^{-1}[M,\partial M] \rangle \\
&=& \langle [(B_\rho \circ B_j)^*\overline{c\beta}_d], EM_d^{-1}[M,\partial M] \rangle \\
&=& \langle c_\rho \overline{v}_d, EM_d^{-1}[M,\partial M] \rangle = c_\rho \mathrm{Vol}(N),
\end{eqnarray*}}
where $B_i : B\mathrm{SL}(n,\mathbb{C}) \ra
B\mathrm{SL}(n,\mathbb{C})^\mathrm{fb}$ is the natural inclusion
map. Therefore, we complete the proof.
\end{proof}

Due to $\langle b_d, \gamma(N) \rangle = c_\rho
\mathrm{Vol}(N)\neq 0$, it can be checked that $\gamma(N)$ is a
nontrivial element in $K_d(\overline{\bq})\otimes \bq$ if $\rho^*b_d\not=0$. Kuessner
\cite[Theorem 3]{Ku} characterizes a complete list of irreducible symmetric
spaces $G/K$ of noncompact type and fundamental representation
$\rho : G \ra \mathrm{GL}(n,\mathbb{C})$ with $\rho^*b_{2m-1}\neq
0$ for $d=2m-1=\dim (G/K)$. In the noncompact case, he only gets
invariants for hyperbolic manifolds since there exist no such
fundamental representations for the other $\br$-rank $1$
semisimple Lie groups. However, in the list, there are a lot of
fundamental representations for higher rank semisimple Lie groups.
Theorem \ref{thm:Ktheory} enables us to get invariants for
$\bq$-rank $1$ locally symmetric spaces, including hyperbolic
manifolds, by using the fundamental representations in the list.

\section{Relation to classical Bloch group}
In \cite{NY}, Neumann-Yang constructed an invariant of finite
volume hyperbolic 3-manifolds which lie in the Bloch group $\cal
B(\bc)$. In this section, we give an invariant of a $\bq$-rank $1$
locally symmetric space, which coincides with the classical Bloch
invariant of cusped hyperbolic $3$-manifolds.

Let $X$ be a symmetric space of noncompact type and
$C_k(\partial_\infty X)$ the free abelian group generated by
$(k+1)$-tuples of points of $\partial_\infty X$ modulo the
relations
\begin{enumerate}
\item
    $(\theta_0,\ldots,\theta_k)=sign(\tau)(\theta_{\tau(0)},\ldots,\theta_{\tau(k)})$
    for any permutation $\tau$ and
\item $(\theta_0,\ldots,\theta_k)=0$ whenever $\theta_i=\theta_j$ for some $i\neq j$
\item $\partial(\theta_0,\ldots,\theta_k)=\sum_{i=0}^k
(-1)^i(\theta_0,\ldots,\hat{\theta}_i,\ldots,\theta_k)$
\end{enumerate}
The \emph{generalized pre-Bloch group of $X$} is
$$\cal P_*(X)=H_*(C_*(\partial_\infty X)\otimes_{\bz G} \bz,
\partial\otimes_{\bz G} id)$$ and the \emph{generalized pre-Bloch group of
$\bc$} as $$\cal P_*^n(\bc)=\cal P_*(\mathrm{SL}(n,\bc)/\mathrm{SU}(n)).$$  Note that $\cal
P_3(H^3_\br)=\cal P_3^2(\bc)$ is the classical pre-Bloch group $\cal P(\bc)$.

\begin{remark} One may wonder what happens if Conditon (1) is omitted. Let $\widehat{C}_*(\partial_\infty X)$ be the chain complex analogously defined without condition (1) and $\pi:\widehat{C}_*(\partial_\infty X)\rightarrow C_*(\partial_\infty X)$ the projection, then (because of condition (ii)) each element of $C_n(\partial_\infty X)$ has $(n+1)!$ preimages and we may define a right-inverse to the projection $\pi$ by sending each $c\in C_n(\partial_\infty X)$ to the formal sum of its preimages divided by $(n+1)!$. One checks that this defines a chain map. Thus one obtains an injection $H_*(C_*(\partial_\infty X)\otimes_{\bz G} \bz)\rightarrow H_*(\widehat{C}_*(\partial_\infty X)\otimes_{\bz G} \bz)$. Therefore, each $G$-equivariant cocycle on $\widehat{C}_*(\partial_\infty X)$ also defines a $G$-equivariant cocycle on $C_*(\partial_\infty X)$. (One may think of this new cocycle as taking the signed average of the evaluations of the old cocycle on the $(n+1)!$ simplices obtained by permuting the vertices.) In particular, for ${\br}$-rank one spaces, the algebraic volume $algvol$ yields a well-defined cocycle on $C_*(\partial_\infty X)\otimes_{\bz G} \bz$. (This was not made explicit in \cite{Kue}.)
\end{remark}

If $\rho:G\ra \mathrm{SL}(n,\bc)$ is a nontrivial (hence reductive) representation,
then it induces a smooth map $X=G/K\ra \mathrm{SL}(n,
\bc)/\mathrm{SU}(n)$ and its extension to the ideal boundary
$$\rho_\infty:\partial_\infty X\ra \partial_\infty
(\mathrm{SL}(n,\bc)/\mathrm{SU}(n)).$$  Using this one can define
a generalized Bloch invariant of a $\bq$-rank $1$ locally
symmetric space $N=\Gamma \backslash G/K$ as follows. Fix a point
$c_0 \in \partial_\infty X$ and define
$ev_{\Gamma,c_0,c_1,\ldots,c_s} : C_*(B\Gamma^{comp}) \ra
C_*(\partial_\infty X)\otimes_{\bz G} \bz$ on generators by
$$ev_{\Gamma,c_0,c_1,\ldots,c_s}(\gamma_1,\cdots,\gamma_k)=(c_0,\gamma_1c_0,\cdots,\gamma_1\cdots
\gamma_k c_0)\otimes 1$$ for $\gamma_1,\ldots,\gamma_k \in
\Gamma$, and
$$ev_{\Gamma,c_0,c_1,\ldots,c_s}(p_1,\ldots,p_{k-1},c_i)=(c_0,p_1
c_0,\ldots p_1\ldots p_{k-1}c_0,c_i)\otimes 1$$ for
$p_1,\ldots,p_{k-1} \in \Gamma_i$. It is straightforward to check
that $ev_{\Gamma,c_0,c_1,\ldots,c_s}$ extends linearly to a chain
map and thus, it induces a homomorphism
$$(ev_{\Gamma,c_0,c_1,\ldots,c_s})_* : H_*^{simp}(B\Gamma^{comp},\bz)\ra \mathcal{P}_*(X).$$

In addition, one can define a map $$\rho_\infty :
C_*(\partial_\infty X) \otimes_{\bz G} \bz \ra C_*( \partial_\infty
(\mathrm{SL}(n,\bc)/\mathrm{SU}(n)))\otimes_{\bz
\mathrm{SL}(n,\bc)} \bz$$ defined by $$ \rho_\infty
((\theta_0,\ldots,\theta_k) \otimes
1)=(\rho_\infty(\theta_0),\ldots,\rho_\infty(\theta_k))
\otimes1.$$ Hence, it yields a homomorphism
$$(\rho_\infty)_* : \mathcal{P}_*(X) \ra \mathcal{P}_*^n(\bc).$$

To obtain an element of the generalized pre-Bloch group of $\bc$
for a $\bq$-rank $1$ locally symmetric space $N=\Gamma \backslash
G/K$, we need an integer homology class in
$H^{simp}_d(B\Gamma^{comp},\bz)$. Note that $EM^{-1}_d[M,\partial
M] \in H^{simp}_d(B\Gamma^{comp},\bq)$ is a rational homology
class.  However, it can be easily checked that we can obtain an
integer homology class $EM^{-1}_d[M,\partial M]_\bz \in
H^{simp}_d(B\Gamma^{comp},\bz)$ from the relative fundamental
class $[M,\partial M]_\bz \in H^d(M,\partial M, \bz)$ with integer
coefficients as follows. Given a relative fundamental cycle $z$
with integer coefficients,
$c(z^0)=z^0+(-1)^{d+1}\mathrm{Cone}_g(\partial z^0)$ defined in
the proof of Lemma \ref{volume2} is also a cycle with integer
coefficients. For another relative fundamental cycle $\bar z$ with
integer coefficients, there are chains $w \in C_{d+1}(M,\bz)$ and
$y \in C_d(\partial M,\bz)$ with all vertices in $x_0$ and such that $z^0
- \bar z^0=\partial w^0+y^0$. Then {\setlength\arraycolsep{2pt}
\begin{eqnarray*}
c(z^0)-c(\bar z^0)&=&z^0-\bar z^0+(-1)^{d+1}(\mathrm{Cone}_g(\partial z^0)-\mathrm{Cone}_g(\partial \bar z^0)) \\
&=& \partial w^0 +y^0 + (-1)^{d+1} \mathrm{Cone}_g(\partial y^0) \\
&=& \partial w^0+y^0+(-1)^{d+1}(\partial \mathrm{Cone}_g(y^0)+(-1)^d y^0) \\
&=& \partial(w^0+(-1)^{d+1} \mathrm{Cone}_g(y^0))
\end{eqnarray*}}
It is obvious that $\Phi(w^0+(-1)^{d+1} \mathrm{Cone}_g(y^0))$ is
a chain in $C_{d+1}^{simp}(B\Gamma^{comp},\bz)$. Hence $c(z^0)$
determines a homology class in $H_d^{simp}(B\Gamma^{comp},\bz)$
independent of the choice of relative fundamental cycle $z$,
denoted by $EM^{-1}_d[M,\partial M]_\bz$.

\begin{defi}
For a $\bq$-rank $1$ locally symmetric space $N$ of dimension $d$,
define an element $\beta_\rho(N)$ in the generalized pre-Bloch group $\mathcal P^n_d(\bc)$ by
$$\beta_\rho(N):= (\rho_\infty)_d \circ
(ev_{\Gamma,c_0,c_1,\ldots,c_s})_d \circ EM^{-1}_d[M,\partial
M]_\bz.$$
\end{defi}

Suppose that every $\rho(\Gamma_i)$ is unipotent. Then the proof
of Proposition \ref{prop:preimage} works for $(B_\rho \circ
B_j)_*EM^{-1}_d[M,\partial M]_\bz \in
H^{simp}_d(B\mathrm{SL}(n,\overline{\bq})^\mathrm{fb},\bz)$. Thus,
it has a preimage $\overline{\gamma}(N)_\bz \in
H^{simp}_d(B\mathrm{SL}(n,\overline{\bq}),\bz)$. In the case that
$N$ is a $\br$-rank $1$ locally symmetric space, Kuessner
\cite{Kue} showed that
$$(ev_{\mathrm{SL}(n,\bc)})_d(\overline{\gamma}(N)_\bz)=\beta_\rho(N)$$
where the evaluation map $$ev_{\mathrm{SL}(n,\bc)} :
C^{simp}_*(B\mathrm{SL}(n,\mathbb{C}),\bz) \ra C_*(\partial_\infty
(\mathrm{SL}(n,\bc)/\mathrm{SU}(n)) \otimes_{\bz
\mathrm{SL}(n,\bc)} \bz$$ is defined on generators by
$$ev(g_1,\ldots,g_k) =(c_0,g_1 c_0,\ldots, g_1 \ldots g_k c_0)\otimes
1.$$

In the same way, this holds for $\bq$-rank $1$ locally
symmetric spaces. Specially
it recovers the classical Bloch invariant of cusped hyperbolic
3-manifolds. For this reason we will call $\beta_\rho(N)$ the
\emph{generalized Bloch invariant} for either a compact locally symmetric
manifold or a finite volume $\bq$-rank $1$ locally symmetric space $N$.

One advantage of our approach is that one can define the Bloch
invariant without the notion of degree one ideal triangulation.
Neumann and Yang used the fact that cusped hyperbolic
$3$-manifolds admit a degree one ideal triangulation to define the
classical Bloch invariant of cusped hyperbolic $3$-manifolds.
Since it is not known whether general locally symmetric spaces admit
such a triangulation, it seems to be difficult to extend the
definition of the classical Bloch invariant from hyperbolic
$3$-manifolds to general locally symmetric spaces in the same way
that Neumann and Yang constructed it. However, we here use only
the relative fundamental cycle $[M,\partial M]_\bz$ to define
$\beta_\rho(N)$ and moreover, this agrees with the classical Bloch
invariant for cusped hyperbolic $3$-manifolds \cite{Kue}. Hence
our approach makes it possible to generalize the Bloch invariant
for cusped hyperbolic $3$-manifolds without the existence of a
degree one ideal triangulation. Furthermore, even if one defines  an
invariant by using a degree one ideal triangulation of $N$, the
invariant should agree with $\beta_\rho(N)$. (This is shown for $\br$-rank 1 spaces in\cite[Theorem 4.0.2]{Kue}.)

\section{Bloch group for convex projective manifolds}
Given a strictly convex real projective manifold $N$, up to taking a
double cover, there is a holonomy map
$$\rho:\pi_1(N)=\Gamma\ra \mathrm{SL}(n,\br)$$ where $\Gamma$ acts on a strictly
convex projective domain $\Omega$ equipped with a Hilbert metric.
If $\Omega$ is conic, then the image of $\rho$ is in
$\mathrm{SO}(n-1,1)$ and the manifold is a real hyperbolic
manifold.

Let's assume that $\mathrm{dim} \left(N\right)=3$.  For a given hyperbolic
representation $\rho_0:\Gamma \ra \mathrm{SL}(2,\bc)\subset
\mathrm{SL}(4,\br)\subset \mathrm{SL}(4,\bc)$, assume there is a
deformation of $\rho_0$ to convex projective structures. The
inclusion $$i:\mathrm{SL}(2,\bc)= \mathrm{SO}(3,1)\subset
\mathrm{SL}(4,\br)\subset \mathrm{SL}(4,\bc)$$ is not the standard
one but by \cite[Corollary 5]{Ku} we have $i^*b_3\neq 0$.

There is a canonical invariant called Bloch invariant developed by
Dupont-Sah and others. We want to generalize this notion to
projective manifolds.

When $\Gamma$ acts on the strictly convex domain $\Omega\subset
\br^3$, in general $Aut(\Omega)$ does not have a Lie group structure
and we do not have a volume form naturally induced from the Lie
group. On the other hands, $\Omega$ has a Finsler metric, called
Hilbert metric invariant under $Aut(\Omega)$. $\Omega$ with a
Hilbert metric behaves like a hyperbolic space, hence one can define
a straight simplex. Fix a base point $x\in\Omega$. The volume class
$v_d\in H^d(\Gamma,\br)$ is defined by the cocycle
\begin{eqnarray}\label{volume}
\nu_d(\gamma_0,\ldots,\gamma_d)=\int_{str(\gamma_0x,\ldots,\gamma_d x)} dvol_\Omega,
\end{eqnarray}
 where
$dvol_\Omega$ is a signed Finsler volume form on $\Omega$, and
$str$ is a geodesic straightening. One can check that it is a
cocycle since $str(\gamma_0x,\ldots,\gamma_d x)$ is a top
dimensional simplex. We have the following lemma similar to Lemma
\ref{lem:IdealStraight}.

\begin{lemma}
The volume of the ideal straight simplex
$str(x_0,\ldots,x_{d-1},c)$ is finite for any $x_0,\ldots,x_{d-1}
\in \Omega$ and $c$ is a cuspidal point.
\end{lemma}
\begin{proof}
This follows from the Proposition 11.2 of \cite{CLT}.
\end{proof}
This lemma allows us to carry out the similar constructions as in
previous sections. Hence we can define a cocycle $\bar \nu_d \in
C^d_{simp}(B\Gamma^{comp})$ extending $comp(\nu_d)$. Let $\bar v_d$
denote the element represented by $\bar \nu_d$ in
$H^d_{simp}(B\Gamma^{comp},\br)$. By Lemma \ref{lem:EMmap},
$H_d(B\Gamma^{comp},\br)=H_d(M,\partial M,\br)=\br$ where $M$ is a
compact manifold with boundary whose interior is homeomorphic to
$N$. It is known that for a cusped hyperbolic 3-manifold, there
exists a degree one  ideal triangulation. We can use the same ideal
triangulation to obtain a triangulation of $M$ by Hilbert metric
ideal tetrahedra to obtain
$$\langle \bar v_d, EM^{-1}_d[M, \partial M]\rangle= \mathrm{Vol}_{Finsler}(N).$$

For the Borel class $b_3\in H^3_c(\mathrm{GL}(\bc),\br)$, it is
known that $\rho^*b_3\neq 0$ for any nontrivial representation
$\rho:\mathrm{SL}(2,\bc)\ra \mathrm{GL}(\bc)$. Then for any finite
volume hyperbolic manifold $N=\Gamma\backslash {H}_\br^3$, the
induced representation $$\rho_0:\Gamma\ra \mathrm{GL}(\bc)$$ gives
rise to a nontrivial Borel class $\rho_0^* b_3=c_{\rho_0}\overline
v_{\rho_0}$ by Lemma \ref{borelclass}. Since a strictly convex
projective structure $\rho:\Gamma \ra \mathrm{SL}(4,\br)$ lies in
the same component containing $\rho_0$ in the character variety
$\chi(\Gamma,\mathrm{SL}(4,\bc))$, $H(\rho)[M,\partial M]$ is
nontrivial, indeed equal to $H(\rho_0)[M,\partial M] \in
H_3(B\mathrm{GL}(\bc)^\delta,\bz)$. Hence
{\setlength\arraycolsep{2pt}
\begin{eqnarray*}
\langle b_3, B(\rho) \circ EM^{-1}[M,\partial M] \rangle &=& \langle b_3, B(\rho_0) \circ EM^{-1}[M,\partial M] \rangle \\
&=& \langle \rho_0^* b_3, EM^{-1}[M,\partial M]\rangle \\
&=& c_{\rho_0}\mathrm{Vol}_{hyp}(M)=c_\rho \mathrm{Vol}_{Finsler}(M).
\end{eqnarray*}}
Since $H^d_{simp}(B\Gamma^{comp},\br)=H_d(B\Gamma^{comp},\br)^*=\br$, $\rho^* b_3=c_\rho \bar v_\rho$.

If $\sum a_i\tau_i$ is  a proper ideal fundamental cycle (i.e.
each simplex is a proper ideal straight simplex) of $M$, the sum
of the cross-ratios $$\beta(M)=\sum a_i[cr(\tau_i)]\in\cal
P_d(\Omega):=H_3(C_*(\partial\Omega)_\Gamma)$$ defines a generalized Neumann-Yang invariant. Note
that if $\Omega$ is not conic, the cross-ratio is not a complex number as in $H_\br^3$. It
is a question how to interpret this invariant in terms of a volume
and the Chern-Simon invariant.

\section{Bloch invariant in $\mathrm{SU}(2,1)$ and $\mathrm{SL}(3,\bc)$}

\subsection{Falbel-Wang invariant}
Falbel constructed a discrete representation
$$\rho:\pi_1(S^3\setminus K)\ra \mathrm{SU}(2,1)\subset \mathrm{SL}(3,\bc)$$ which is faithful and parabolic on the
torus boundary where $K$ is a figure 8 knot. Falbel-Wang further
constructed a Bloch invariant using this representation. In this section, we prove
that their invariant can be computed from $H\left(\rho\right)EM^{-1}\left[M,\partial M\right]$.

Following \cite[Section 3.8]{bfg} we identify the complex hyperbolic space $H_\bc^2$ with
$\pi\left(V_-\right)\subset {\mathbb C}P^2$, where $V_-:=\left\{\left(
x,y,z\right)\in{\mathbb C}^3: x\overline{z}+
y\overline{y}+z\overline{x}<0\right\}$ and $\pi:{\mathbb C}^3-0\rightarrow {\mathbb C}P^2$ is the canonical projection.
The ideal boundary $\partial_\infty{\mathbb C}H^2\simeq S^3$ is then identified with $\pi\left(V_0\right)$, where $V_0:=
\left\{\left(x,y,z\right)\in{\mathbb C}^3: x\overline{z}+y\overline{y}+z\overline{x}=0\right\}$.

The following construction involves an identification of ${\mathbb C}P^1$ with the set of
(affine) complex lines through a given point in ${\mathbb C}P^2$. There is some arbitrariness
in choosing such an identification, a specific choice is given in \cite[Section 2.5]{FW} with a derived explicit formula in \cite[Definition 2.24]{FW}.

\begin{defi}\label{falbelwang}{\bf (Falbel-Wang construction, \cite[Section 2.5]{FW})}:

a) Fix an identification $\partial_\infty H^2_{\mathbb C}=S^3$ and
define a homomorphism $$FW_{01}:{\mathcal P}_3\left( H_{\mathbb C}^2\right)\rightarrow
{\mathcal P}_3\left(
 H_{\mathbb R}^3\right)$$
on generators $\left(p_0,p_1,p_2,p_3\right)$ of ${\mathcal P}_3\left(
 H_{\mathbb C}^2\right)$ by $$FW_{01}\left(p_0,p_1,p_2,p_3\right)=\left(t_0,t_1,t_2,t_3\right),$$
where we define $t_0\in {\mathbb C}P^1=\partial_\infty H^3_{\mathbb R}$ to be
the complex line through $p_0$ tangent to $S^3\subset {\mathbb C}P^2$ and for $i=1,2,3$ we define
$t_i\in {\mathbb C}P^1=\partial_\infty H^3_{\mathbb R}$ to be the complex line in ${\mathbb C}P^2$ passing through $p_0$ and $p_i$.

b) For $a\not=b\in\left\{0,1,2,3\right\}$ there are unique $k,l\in
\left\{0,1,2,3\right\}$ such that the ordered
set $\left(a,b,k,l\right)$ is an even permutation of $\left(0,1,2,3\right)$ and we define $$FW_{ab}:
{\mathcal P}_3\left( H_{\mathbb C}^2\right)\rightarrow
{\mathcal P}_3\left(
 H_{\mathbb R}^3\right)$$
on generators $\left(p_0,p_1,p_2,p_3\right)$ of ${\mathcal P}_3\left(
 H_{\mathbb C}^2\right)$ by $$FW_{ab}\left(p_0,p_1,p_2,p_3\right)=FW_{01}\left(p_a,p_b,p_k,p_l\right).$$
\end{defi}
We will occasionally consider the $FW_{ab}$ as maps which send
($\mathrm{SU}(2,1)$-orbits of) ideal simplices in $H^2_{\mathbb
C}$ to ($\mathrm{SO}(3,1)$-orbits of) ideal simplices in
$H^3_{\mathbb R}$.

If $M$ is a hyperbolic $3$-manifold and $\rho:\pi_1M\rightarrow
\mathrm{SU}(2,1)$ is a reductive representation (that means
$\rho\left(\pi_1M\right)\subset \mathrm{SU}(2,1)$ is a reductive
subgroup), then by \cite{cor} (see also \cite{lab}) we obtain a
$\rho$-equivariant developing map
$D:H^3_{\mathbb{R}}=\widetilde{M}\rightarrow H^2_{\mathbb C}$ and
a continuous boundary map $\partial_\infty D:\partial_\infty
H^3_\br \rightarrow \partial_\infty H^2_{\mathbb C}$. In
particular, if $T$ is an ideal simplex in $M$ (that is a
$\pi_1M$-orbit of some ideal simplex $\widetilde{T}$ with vertices
$v_0,v_1,v_2,v_3\in\partial_\infty H_\br^3$), then we can define
$D\left(T\right)$ to be the $\pi_1M$-orbit of the ideal simplex in
$H^2_{\mathbb C}$, whose vertices are $\partial_\infty D
\left(c_i\right)$ for $i=0,1,2,3$.

Let ${\mathcal P}_3^{nd}
\left(H^3_{\mathbb R}\right)\subset {\mathcal P}_3
\left(H^3_{\mathbb R}\right)$ be the subgroup generated by 4-tuples $\left(z_0,z_1,z_2,z_3\right)$ of {\em pairwise distinct} points. (By definition, simplices in
an ideal triangulation are nondegenerate and thus give elements in ${\mathcal P}_3^{nd}
\left(H^3_{\mathbb R}\right)$.)

Recall that the cross ratio $$X:{\mathcal P}_3^{nd}
\left(H^3_{\mathbb R}\right)\rightarrow
{\mathcal{P}}\left({\bc}\right)$$ is well defined and yields an
isomorphism between ${\mathcal P}_3^{nd} \left(H^3_{\mathbb
R}\right)$ and ${\mathcal{P}}\left({\bc}\right)$. We will use the
abbreviation
$$FW:=FW_{01}+FW_{10}+FW_{23}+FW_{32}:{\mathcal P}_3\left(
H_{\mathbb C}^2\right)\rightarrow {\mathcal P}_3\left(
 H_{\mathbb R}^3\right).$$

\begin{defi}\label{beta}{\bf (Falbel-Wang invariant)}: Let $M=\cup_{i=1}^r T_i$
be an ideal triangulation of a hyperbolic $3$-manifold and
$\rho:\pi_1M\rightarrow \mathrm{SU}(2,1)$ a reductive
representation, then the invariant
$\beta_{FW}\left(M\right)\in{\mathcal{B}}\left({\mathbb
C}\right)\subset {\mathcal{P}}\left({\mathbb C}\right)$ is defined
by $$\beta_{FW}\left(M\right):=\sum_{i=1}^r X\left(FW\left(cr
\left(D\left(T_i\right)\right)\right)\right).$$
\end{defi}

(It is proved in \cite[Theorem 1.1]{FW} that $\beta_{FW}\left(M\right)$ lies in ${\mathcal{B}}\left({\mathbb C}\right)$ and does not depend on the ideal triangulation. The latter fact will also follow from our \hyperref[falbelwangcomparison]{Lemma \ref*{falbelwangcomparison}}.)


\begin{lemma}\label{falbelwangcomparison}Let $M$ be a finite-volume hyperbolic $3$-manifold, $\rho:\pi_1M\rightarrow \mathrm{SU}(2,1)$ a reductive representation which maps the peripheral subgroups to unipotent subgroups. Then
$$\beta_{FW}\left(M\right)= X\left(FW\left(H\left(ev\right)\left(H\left(\rho\right)\left(EM^{-1}\left[M,\partial M\right]\right)\right)\right)\right).$$
\end{lemma}
\begin{proof}
Let $x_0\in M$, $c_0\in\partial_\infty\widetilde{M}=\partial_\infty H^3_{\mathbb{R}}$ and $b_0:=\partial_\infty D\left(c_0\right)\in\partial_\infty H^2_{\mathbb{C}}$. Lemma 3.3.4 in \cite{Kue} constructs a chain map $\hat{C}:\hat{C}_*^{str,x_0}\left(M\right)\rightarrow \hat{C}_*^{str,c_0}\left(M\right).$

Let $G=\mathrm{SU}(2,1)$, $K=\mathrm{S}(\mathrm U(2)\times \mathrm U(1))$, $\Gamma=\pi_1M$, $\Gamma_i=\pi_1\partial_iM$ for the path components $\partial_iM$ of $\partial M$, $c_i$ the cusps associated to $\Gamma_i$.
We use the commutative diagram
$$\begin{xy}
\xymatrix{
C_*\left(BG^{comp}\right)
\ar[r]^{ev_{b_0}}&C_*\left(\partial_\infty G/K\right)_G\\
C_*\left(B\Gamma^{comp}\right)\ar[r]^{ev_{c_0}}\ar[u]^{B\rho}&
C_*\left(\partial_\infty\widetilde{M}\right)_\Gamma\ar[u]^{\partial_\infty D}\\
\hat{C}_*^{str,x_0}\left(M\right)\ar[r]^{\hat{C}}\ar[u]^{\hat{\Phi}}& \hat{C}_*^{str,c_0}\left(M\right)\ar[u]^{cr}\\
C_*\left(M\cup\left\{\Gamma c_1,\ldots,\Gamma c_s\right\}\right)\ar[u]^{str}&\\
Z_*\left(\overline{M},\partial \overline{M}\right)
\rightarrow C_*\left(\mathrm{Dcone}\left(\cup_{i=1}^s \partial_i\overline{M}\rightarrow \overline{M}\right)\right)\ar[u]^{\simeq}&.}
 \end{xy}$$
whose derivation (except for the first square) is explained in the
proof of \cite[Theorem 4.0.2]{Kue}. In the first square we have
$$ev_{c_0}\left(g_1,\ldots,g_n\right)=\left(c_0,g_1c_0,\ldots,g_1\ldots
g_nc_0\right) \otimes 1$$ for $\left(g_1,\ldots,g_n\right)\in BG$
and
$$ev_{c_0}\left(\mathrm{Cone}_{c_i}\left(g_1,\ldots,g_{n-1}\right)\right)
=\left(c_0,g_1c_0,\ldots,g_1\ldots g_{n-1}c_0,c_i\right)\otimes
1,$$ similarly for $ev_{c_0}$.

The diagram shows that
$$H\left(ev_{b_0}\right)\left(H\left(\rho\right)\left(EM^{-1}\left[M,\partial M\right]\right)\right)$$
is represented by $$\partial_\infty
D\left(cr\left(\hat{C}\left(str\left(z+\mathrm{Cone}\left(\partial
z\right)\right) \right)\right)\right)$$ whenever $z\in
Z_*\left(\overline{M},\partial\overline{M}\right)$ is a relative
fundamental cycle.

If $z\in Z_*\left(\overline{M},\partial\overline{M}\right)$ is a
relative fundamental cycle, then
$str\left(z+\mathrm{Cone}\left(\partial z\right)\right)\in
\hat{C}^{str,x_0}_*\left(M\right)$ is an ideal fundamental cycle
in the sense of \cite[Definition 3.1.4]{Kue}. By \cite[Lemma 3.3.4]{Kue}
this implies that
$\hat{C}\left(str\left(z+\mathrm{Cone}\left(\partial
z\right)\right)\right)$ is an ideal fundamental cycle.

On the other hand, let $M=\cup_{i=1}^r T_i$ be an ideal
triangulation. Let $p_0^i,p_1^i,p_2^i,p_3^i\in
\partial_\infty\widetilde{M}=\partial_\infty H^3_{\mathbb R}$ be
the ideal vertices of $T_i$. Then\\ $FW\left(\partial_\infty
D\left(p_0^i\right)\right), $ $ FW\left(\partial_\infty
D\left(p_1^i\right)\right), FW\left(\partial_\infty
D\left(p_2^i\right)\right), FW\left(\partial_\infty
D\left(p_3^i\right)\right)$ are the ideal vertices of $FW\left(
D\left(T_i\right)\right)$. By definition we have
$$\beta_{FW}\left(M\right)=\sum_{i=1}^r4X\left(FW_{01}\left(D\left(T_i\right)\right)
\right).$$ We have proved in the proof of \cite[Lemma 3.4.1]{Kue} that
the homology classes of $cr\left(\hat{C}\left(str\left(
z+\mathrm{Cone}\left(\partial z\right)\right)\right)\right)$ and
of $\sum_{i=1}^r cr\left(T_i \right)$ are the same in
$H_*\left(C_*\left(\partial_\infty\widetilde{M}\right)_\Gamma\right)$.
Thus there is some $w\in C_*\left(
\partial_\infty\widetilde{M}\right)_\Gamma$ with
$$\partial w=cr\left(\hat{C}\left(str\left(z+\mathrm{Cone}\left(\partial z\right)
\right)\right)\right)-\sum_{i=1}^r cr\left(T_i
\right).$$
$\partial_\infty D$ is a chain map by construction.
By \cite[Lemma 3.2]{FW} (following \cite[Theorem 5.2]{falvol}), we have that  $X\left(FW\left(\partial C_4\left(\partial_\infty H^2_\bc\right)_G\right)\right)\subset
{\mathbb Z}\left[{\mathbb C}-\left\{0,1\right\}\right]$ is in the subgroup generated by the 5-term relations, that is
$X\left(FW\left(\partial u\right)\right)=0\in{\mathcal{B}}\left({\mathbb C}\right)$
for all $u\in C_4\left(\partial_\infty H^2_\bc\right)_G$.
Hence we obtain
$$X\left(FW\left(\partial_\infty D\left(cr\left(\hat{C}\left(str\left(z+\mathrm{Cone}\left(\partial z\right)\right)\right)\right)\right)\right)\right)-\sum_{i=1}^r
X\left(FW\left(\partial_\infty D\left(cr\left(T_i\right)\right)\right)\right)$$
$$=
X\left(\partial FW\left(\partial_\infty D\left(w\right)\right)\right)=0.$$
Hence the cycle $\sum_{i=1}^r
X\left(FW\left(\partial_\infty D\left(cr\left(T_i\right)\right)\right)\right)$ represents the homology class
$X\left(FW\left(H\left(ev_{b_0}\right)\left(H\left(\rho\right)\left(EM^{-1}\left[M,\partial M\right]\right)\right)\right)\right)$. Because of $\partial_\infty D\left(cr\left(T_i\right)\right)=cr\left(D\left(T_i\right)\right)$ this implies the claim.
\end{proof}

For \hyperref[corfw]{Corollary \ref*{corfw}} and
\hyperref[corbfg]{Corollary \ref*{corbfg}} we will consider the
situation that a $3$-manifold $M^\tau$ is obtained from another
$3$-manifold $M$ by cutting along some $\pi_1$-injective surface
$\Sigma\subset M$ and regluing via $\tau:\Sigma\rightarrow\Sigma$.
If in this situation for a representation $\rho:\pi_1M\rightarrow
G$ we have some $A\in G$ with
$\rho\left(\tau_*h\right)=A\rho\left(h\right)A^{-1}$ for all
$h\in\pi_1\Sigma$, then we get an induced representation
$\rho^\tau:\pi_1M^\tau\rightarrow G$ by a standard application of
the Seifert-van Kampen Theorem as in \cite[Section 2]{Ku2}. This
representation $\rho^\tau$ will be used in the statements of
\hyperref[corfw]{Corollary \ref*{corfw}} and
\hyperref[corbfg]{Corollary \ref*{corbfg}}. We say that the
representation is parabolics-preserving if it sends $\pi_1\partial
M$ to parabolic elements. (It is easy to see from the explicit
description in the proof of \cite[Proposition 3.1]{Ku2} that
$\rho^\tau$ is reductive and parabolics-preserving if $\rho$ is.)
\begin{cor}\label{corfw}Let $M$ be a compact, orientable $3$-manifold,
$\Sigma\subset M$ a properly embedded, incompressible,
boundary-incompressible, $2$-sided surface,
$\tau:\Sigma\rightarrow\Sigma$ an orientation-preserving
diffeomorphism of finite order and $M^\tau$ the manifold obtained
by cutting $M$ along $\Sigma$ and regluing via $\tau$.

If $M$ and $M^\tau$ are hyperbolic and if the reductive,
parabolics-preserving representation $\rho:\pi_1M\rightarrow
\mathrm{SU}(2,1)$ satisfies
$\rho\left(\tau_*\sigma\right)=A\rho\left(\sigma\right)A^{-1}$ for
some $A\in \mathrm{SU}(2,1)$ and all $\sigma\in\pi_1\Sigma$, then
$$\beta_{FW}\left(M\right)\otimes
1=\beta_{FW}\left(M^\tau\right)\otimes 1\in
{\mathcal{B}}(\bc)\otimes\bq$$ with respect to the representations
$\rho$ and $\rho^\tau$.\end{cor}
\begin{proof}The proof is essentially the same as the one of \hyperref[corbfg]{Corollary \ref*{corbfg}} below,
which in turn is essentially the same as that for \cite[Theorem
1]{Ku2}. Therefore we omit the proof at this point and just
mention that literally the same argument (just replacing
$\mathrm{SL}(3,\bc)$ by $\mathrm{SU}(2,1)$) as given below in the
proof of \hyperref[corbfg]{Corollary \ref*{corbfg}} shows that
$$H\left(\rho\right)\left(EM^{-1}\left[M,\partial M\right]_{\bq}\right)=H\left(\rho^\tau\right)\left(EM^{-1}\left[M^\tau,\partial M^\tau\right]_{\bq}\right)$$ and in view of \hyperref[falbelwangcomparison]{Lemma \ref*{falbelwangcomparison}} this implies $\beta_{FW}\left(M\right)\otimes 1=\beta_{FW}\left(M^\tau\right)\otimes 1$.\end{proof}
\subsection{Tetrahedra of flags}
\begin{defi}\label{flags} {\bf (\cite[Section 2]{bfg})}: Let
$${\mathcal{F}}l\left({\mathbb C}\right)=
\left\{\left(\left[x\right],\left[f\right]\right)\in {{P}}\left({\mathbb C}^3\right)\times {{P}}\left({\mathbb C}^{3*}\right): f\left(x\right)=0\right\}$$ where $P(V)$ denotes the projectivization of $V$.
For $$T=\left(\left(\left[x_0\right],\left[f_0\right]\right),
\left(\left[x_1\right],\left[f_1\right]\right), \left(\left[x_2\right],\left[f_2\right]\right),
\left(\left[x_3\right],\left[f_3\right]\right)\right)\in C_3\left(\fl l\left({\mathbb C}\right)\right)$$
and $a\not=b\in \left\{0,1,2,3\right\}$ we define $z_{ab}\in{\mathbb C}$ as follows:
choose $k,l\in\left\{0,1,2,3\right\}$ such that $\left(a,b,k,l\right)$ is a positive permutation of $\left(0,1,2,3\right)$ and let $$
z_{ab}:=\frac{f_a\left(x_k\right)det\left(x_a,x_b,x_l\right)}{f_a\left(x_l\right)det\left(x_a,x_b,x_k\right)}.$$
Then define $$\beta:C_3\left(\fl
l\right)\rightarrow {\mathcal{P}}\left({\mathbb C}\right)$$
by $$\beta\left(T\right):=\left[z_{01}\right]+\left[z_{10}\right]+\left[z_{23}\right]+\left[z_{32}\right].$$
\end{defi}

Let $H_3\left(\fl l\right):=H_3\left(C_*\left(\fl
l\right)_G\right)$ for the canonical action of
$G:=\mathrm{SL}(3,\bc)$ on $\fl l$. Then
\cite[Proposition 3.3]{bfg} implies that $\beta$ yields a
well-defined map $\beta_*:H_3\left(\fl
l\right)\rightarrow{\mathcal{P}}\left(\bc\right)$.

Moreover, if $M=\Gamma\backslash H^3_{\br}$ is a hyperbolic
$3$-manifold and $h:{\bc}P^1\rightarrow \fl l$ a map equivariant
with respect to {\em some} homomorphism
$\Gamma\rightarrow
\mathrm{SL}(3,{\bc})$, then one obtains a well-defined
chain map $h_*:C_*\left({\bc}P^1\right)_\Gamma\rightarrow
C_*\left(\fl l\right)_{\mathrm{SL}(3,{\bc})}$.

The following definition is due to Bergeron-Falbel-Guilloux (\cite{bfg}).

\begin{defi}\label{bfgbloch}If $M=\cup_{i=1}^r T_i$ is an ideal triangulation
of a hyperbolic $3$-manifold, $\rho:\pi_1M\rightarrow
\mathrm{SL}(3,\bc)$ a representation and
$$h:{\mathbb C}P^1\rightarrow \fl
l\left({\mathbb C}\right)$$ a $\rho$-equivariant map, then define
$$\beta_h\left(M\right):=\sum_{i=1}^r
\beta_*\left(h_*\left(P_0^i, P_1^i, P_2^i,
P_3^i\right)\right)\in{\mathcal{P}}\left(\bc\right),$$ where
$P_0^i,P_1^i,P_2^i,P_3^i$ are the vertices of $T_i$.\end{defi} We
remark (compare \cite[Proposition 10.79]{bh}) that $\fl l$ corresponds to the set $\mathrm{SL}(3,\bc)/P$ of Weyl chambers in
$\partial_\infty (\mathrm{SL}(3,\bc)/\mathrm{SU}(3))$ where $P$ is a minimal parabolic subgroup. If
$\rho:\pi_1M\rightarrow \mathrm{SL}(3,\bc)$ is a reductive
representation, then there exists a $\rho$-equivariant harmonic
map $H^3_{\br}= \widetilde{M}\rightarrow
\mathrm{SL}(3,\bc)/\mathrm{SU}(3)$ (see \cite{cor},\cite{lab}). 
 But it is not easy to prove the existence of a $\rho$-equivariant boundary map
${\bc}P^1\rightarrow\partial_\infty
(\mathrm{SL}(3,\bc)/\mathrm{SU}(3))$.
It should be easier to show the existence of the boundary map $h:{\bc}P^1\rightarrow \mathrm{SL}(3,\bc)/P=\fl l$ instead.
We will not deal with that issue here, in this paper, we always assume the existence of such a map.\\

{\bf Relation to hyperbolic Bloch invariant,
\cite[Section 3.7]{bfg}.} If $M$ is an orientable hyperbolic
manifold, then by Culler's Theorem its monodromy representation
$\Gamma\rightarrow \mathrm{PSL}(2,\bc)$ lifts to
$\mathrm{SL}(2,\bc)$. Composition with the (unique) irreducible
representation $\mathrm{SL}(2,\bc)\rightarrow \mathrm{SL}(3,\bc)$
yields representations $\rho:\Gamma\rightarrow
\mathrm{SL}(3,\bc)$. In this case there is a canonically
(independent of $\Gamma$) defined $\rho$-equivariant map
${\bc}P^1\rightarrow\fl l$ as follows.

Recall that the irreducible 3-dimensional representation of
$\mathrm{SL}(2,\bc)$ can be defined as follows. Consider the
$\bc$-vector space of complex homogeneous polynomials of degree 2
in two variables. This is a 3-dimensional vector space $V$
generated by $x^2, xy$ and $y^2$. $\mathrm{SL}(2,\bc)$ acts by
$(AP)(x,y):=P(A^{-1}(x,y))$. We may consider its projectivization
$P(V)$ and the projectivization of the dual space $P(V^*)$, whose
elements we will write as homogeneous column vectors. Then a
$\rho$-equivariant map $h:{\bc}P^1\rightarrow \fl l\subset
P(V)\times P(V^*)$ is given by
$$h(\left[x,y\right]):=(\left[x^2,xy,y^2\right],\left[\frac{1}{2}y^2,-xy,\frac{1}{2}x^2\right]^T).$$
(\cite{bfg} gives an apparently different construction which however - after computation - yields the same map $h$.) It turns out that the so-defined $\beta_h(M)$ coincides with 4 times the usual Bloch invariant $\beta(M)$ of the hyperbolic $3$-manifold $M$. Indeed, if $T=(P_0,P_1,P_2,P_3)\in C_3({\bc}P^1)$ is an ideal simplex of cross ratio $t$, then explicit computation shows that the simplex $h(T)=(h(P_0),h(P_1),h(P_2),h(P_3))\in C_3(\fl l)$ satisfies $z_{01}(h(T))=z_{10}(h(T))=z_{23}(h(T))=z_{32}(h(T))=t$.\\

{\bf Relation to CR Bloch invariant, \cite[Section
3.8]{bfg}.} If $D$ is the developing map of a reductive
representation $\pi_1M\rightarrow \mathrm{SU}(2,1)$ and $h$ is the
composition of $\partial_\infty D$ with the map $S^3\rightarrow\fl
l \left({\mathbb C}\right)$ given in \cite[Section 3.8]{bfg}, then
$\beta_h\left(M\right)$ coincides with $\beta_{FW}\left(M\right)$.

\begin{lemma}\label{bfgcomparison}Let $M$ be a finite-volume hyperbolic $3$-manifold, and $h:{\mathbb C}P^1\rightarrow \fl l
\left({\mathbb C}\right)$ equivariant with respect to some
homomorphism
$\pi_1M\rightarrow
\mathrm{SL}(3,{\bc})$. Then

$$\beta_{h}\left(M\right)= \beta_*h_*H\left(ev\right)EM^{-1}\left[M,\partial M\right].$$
\end{lemma}
\begin{proof}
Let $c_0\in\partial_\infty H^3_{\br}$. The commutative diagram in
the proof of \hyperref[falbelwangcomparison]{Lemma
\ref*{falbelwangcomparison}} shows that
$H\left(ev_{c_0}\right)EM^{-1}\left[M,\partial M\right]$ is
represented by
$cr\left(\hat{C}\left(str\left(z+\mathrm{Cone}\left(\partial
z\right)\right) \right)\right)$ whenever $z\in
Z_*\left(\overline{M},\partial\overline{M}\right)$ is a relative
fundamental cycle.

In the proof of \hyperref[falbelwangcomparison]{Lemma \ref*{falbelwangcomparison}} we have seen that
the homology classes of $cr\left(\hat{C}\left(str\left(
z+\mathrm{Cone}\left(\partial z\right)\right)\right)\right)$ and
of $\sum_{i=1}^r cr\left(T_i \right)$ are the same in
$H_*\left(C_*\left(\partial_\infty\widetilde{M}\right)_\Gamma\right)$,
whenever $M=\cup_{i=1}^r T_i$ is an ideal triangulation. Thus
there is some $w\in C_*\left(
\partial_\infty\widetilde{M}\right)_\Gamma$ with
$$\partial w=cr\left(\hat{C}\left(str\left(z+\mathrm{Cone}\left(\partial z\right)
\right)\right)\right)-\sum_{i=1}^r cr\left(T_i
\right).$$
$h_*$ is a chain map by construction. Moreover,
by \cite[Proposition 3.3]{bfg} (following from \cite[Theorem 5.2]{falvol}) we have that $\beta$ maps boundaries to zero,
thus $\beta\left(\partial h_*\left(w\right)\right)=0$,
which implies
$$\beta\left(h_*\left(cr\left(\hat{C}\left(str\left(z+\mathrm{Cone}\left(\partial z\right)\right)\right)\right)\right)\right)-\sum_{i=1}^r
\beta\left(h_*\left(cr\left(T_i\right)\right)\right)
=
\beta\left(\partial h_*\left(w\right)\right)=0.$$
Thus the cycle $\sum_{i=1}^r
\beta\left(h_*\left(cr\left(T_i\right)\right)\right)$ represents the homology class
$$\beta_*h_*H\left(ev_{c_0}\right)EM^{-1}\left[M,\partial M\right],$$ which
implies the claim.

\end{proof}
For the following corollary we will use the notations introduced
before in \hyperref[corfw]{Corollary \ref*{corfw}}. The following
corollary applies for example when $M^\tau$ is a (generalized)
mutation of $M$ and $\rho:\pi_1M\rightarrow \mathrm{SL}(3,\bc)$ is
the composition of the inclusion $\pi_1M\subset
\mathrm{SL}(2,\bc)$ with some representation
$\mathrm{SL}(2,\bc)\rightarrow \mathrm{SL}(3,\bc)$. In this case
$\rho^\tau$ is obtained from the composition of the inclusion
$\pi_1M^\tau\subset \mathrm{SL}(2,\bc)$ (see \cite[Section
2]{Ku2}) with the same representation
$\mathrm{SL}(2,\bc)\rightarrow \mathrm{SL}(3,\bc)$ and we have
$h=h^\tau$.
\begin{cor}\label{corbfg}Let $M$ be a compact, orientable $3$-manifold, $\Sigma\subset M$ a properly embedded, incompressible, boundary-incompressible, 2-sided surface, $\tau:\Sigma\rightarrow\Sigma$ an orientation-preserving diffeomorphism of finite order and $M^\tau$ the manifold obtained by cutting $M$ along $\Sigma$ and regluing via $\tau$.

If $M$ and $M^\tau$ are hyperbolic and if the parabolics-preserving representation $\rho:\pi_1M\rightarrow \mathrm{SL}(3,\bc)$ satisfies $\rho\left(\tau_*\sigma\right)=A\rho\left(\sigma\right)A^{-1}$ for some $A\in \mathrm{SL}(3,\bc)$ and all $\sigma\in\pi_1\Sigma$, then $$\beta_{h}\left(M\right)\otimes 1=\beta_{h^\tau}\left(M^\tau\right)\otimes 1\in{\mathcal{B}}(\bc)\otimes\bq$$
when $h$ and $h^\tau$ are $\rho$- resp.\  $\rho^\tau$-equivariant maps from $\partial_\infty { H}_\br^3$ to ${\mathcal{F}}l$.\end{cor}
\begin{proof}The proof is essentially the same as for \cite[Theorem 1]{Ku2}. Since $\Sigma$ is a 2-sided, properly embedded surface, it has a
neighborhood $N\simeq\Sigma\times\left[0,1\right]$ in $M$, and a
neighborhood $N^\tau\simeq\Sigma\times\left[0,1\right]$ in
$M^\tau$. The complements $M-int\left(N\right)$ and
$M^\tau-int\left(N^\tau\right)$ are diffeomorphic and we let $X$
be the union of $M$ and $M^\tau$ along this identification of
$M-int\left(N\right)$ and $M^\tau-int\left(N^\tau\right)$. The
union of $N$ and $N^\tau$ yields a copy of the mapping torus
$T^\tau$ in $X$. We have
\begin{equation}\label{a}i_{M*}\left[M,\partial M\right]-i_{M^\tau
*}\left[M^\tau,\partial M^\tau\right]= i_{T^\tau
*}\left[T^\tau,\partial T^\tau\right]\in H_3 (X,\partial X,\mathbb
Z).\end{equation} The made assumption implies that $\rho$ and
$\rho^\tau$ extend to a representation $\rho_X:\pi_1X\rightarrow
\mathrm{SL}\left(3,\bc\right)$.

As in the proof of \cite[Theorem 1]{Ku2} we have a finite cyclic
covering $\widehat{X}\rightarrow X$ such that $\widehat{X}$
contains a copy of $\Sigma\times S^1$ finitely covering
$T^\tau\subset X$. Let
$\widehat{M},\widehat{M^\tau}\subset\widehat{X}$ be the preimages
of $M$ and $M^\tau$. Application of the transfer map yields
\begin{equation}\label{tr}
i_{\widehat{M}*}\left[\widehat{M},\partial \widehat{M}\right]-
i_{\widehat{M}^\tau *}\left[\widehat{M}^\tau,\partial
\widehat{M}^\tau \right]= i_{\Sigma\times{\bf S}^1
*}\left[\Sigma\times{\bf S}^1,\partial\Sigma\times{\bf
S}^1\right].\end{equation} Again as in the proof of \cite[Theorem
1]{Ku2} we obtain a representation
$\rho_{\widehat{X}}:\pi_1\widehat{X}\rightarrow
\mathrm{SL}(3,\bc)$ and - because the lift $\widehat{X}$ is chosen
such that
$\rho_{\widehat{X}}\left(\pi_1\partial\widehat{X}\right)$ consists
of parabolics - a continuous map
$$B\rho_{\widehat{X}}:\left(B\pi_1\widehat{X}\right)^{comp}\rightarrow
B\mathrm{SL}(3,\bc)^{comp}.$$ The classifying map
$\Psi_{\widehat{X}}:\widehat{X}\rightarrow | B\pi_1{\widehat{X}}|$
extends to
$$\Psi_{\widehat{X}}:\mathrm{Dcone}\left(\cup_{i=1}^s\partial_i\widehat{X}\rightarrow
\widehat{X}\right) \rightarrow
|\left(B\pi_1{\widehat{X}}\right)^{comp}|$$ and the same argument
as in \cite{Ku2} shows that $\left(|
B\rho_{\widehat{X}}|\Psi_{\widehat{X}}i_{\Sigma\times{\bf
S}^1}\right)_*$ factors over
$H_3\left(\Sigma,\partial\Sigma\right)=0$ and is therefore $0$.
Thus \hyperref[tr]{Equation \ref*{tr}} implies
{\setlength\arraycolsep{2pt}
\begin{eqnarray*}
\left(| B\rho_{\widehat{X}}|\Psi_{\widehat{X}}i_{\widehat{M}}\right)_*
\left[\widehat{M},\partial\widehat{M}\right] &=& \left(
(| B\rho_{\widehat{X}}|\Psi_{\widehat{X}}
i_{\widehat{M}^\tau}\right)_*\left[\widehat{M}^\tau,\partial \widehat{M}^\tau\right] \\
&\in& H_3\left(| B\mathrm{SL}\left(3,\bc\right)^{comp}|\right).\end{eqnarray*}}

Again following the same argument from \cite{Ku2} we conclude
$$H\left(\rho\right)\left(EM^{-1}\left[M,\partial
M\right]_{\bq}\right)=H\left(\rho^\tau\right)\left(EM^{-1}\left[M^\tau,\partial
M^\tau\right]_{\bq}\right).$$ The $\rho$-equivariance of $h$
implies that $h\circ
ev_{\mathrm{SL}(2,\bc)}=ev_{\mathrm{SL}(3,\bc)}\circ \rho$, hence
$h_*H\left(ev\right)=H\left(ev\right)H\left(\rho\right)$ and thus
{\setlength\arraycolsep{2pt}
\begin{eqnarray*}
\beta_h\left(M\right)\otimes
1&=&\beta_*h_*H\left(ev\right)EM^{-1}\left[M,\partial M\right]_{\bq} \\
&=& \beta_*H(ev)H(\rho)\left(EM^{-1}\left[M,\partial M\right]_{\bq}\right) \\
&=& \beta_*H(ev)H\left(\rho^\tau\right)\left(EM^{-1}\left[M^\tau,\partial M^\tau\right]_{\bq}\right)\\
&=& \beta_*h_*H\left(ev\right)EM^{-1}\left[M^\tau,\partial M^\tau\right]_{\bq}\\
&=& \beta_h\left(M^\tau\right)\otimes 1
\end{eqnarray*}} in view of \hyperref[bfgcomparison]{Lemma \ref*{bfgcomparison}}.\end{proof}

The so-called Bloch regulator map
$$\rho:{\mathcal{B}}(\bc)\rightarrow \bc/\bq$$
is known to send the Bloch invariant $\beta(M)$ of hyperbolic
$3$-manifolds to
$\frac{i}{2\pi^2}(\mathrm{Vol}(M)+i\mathrm{CS}(M))\ mod\ \bq$, as
was proved in \cite[Theorem 1.3]{NY}. In other words, the
imaginary part of $\rho(\beta(M))$ determines the volume (and the
real part determines the Chern-Simons invariant mod $\bq$). Thus
it is natural to define the volume of flag structures as (a
multiple of) the imaginary part of $\rho(\beta_h(M))$.
Bergeron-Falbel-Guilloux in fact define in \cite[Section 3.6]{bfg}
the volume of a flag structure to be
$\frac{2\pi^2}{4}Im(\rho(\beta_h(M))$. The analogously defined
volume of CR structures is always zero by \cite[Theorem 3.12]{FW}
but the volume of flag structures is a nontrivial and potentially
interesting invariant. \hyperref[corfw]{Corollary \ref*{corfw}} of
course implies its invariance under the cut-and-paste operation
described in the statement of the corollary.

\end{document}